\documentclass[twoside]{article}

\usepackage{amssymb,amsmath,amsthm,bbm}
\usepackage{authblk}
\usepackage{mathabx}
\usepackage{mathtools}
\usepackage{color}
\usepackage{graphicx}
\graphicspath{{Figs/}}

\usepackage[pagebackref]{hyperref}
\newcommand{\mres}{%
  \,\raisebox{-.127ex}{\reflectbox{\rotatebox[origin=br]{-90}{$\lnot$}}}\,%
}

\newcommand{\nwc}{\newcommand}
\nwc{\N}{\mathbb{N}}
\nwc{\R}{\mathbb{R}}
\nwc{\E}{\mathbb{E}}
\nwc{\D}{\partial}
\nwc{\calC}{\mathcal{C}}
\nwc{\calH}{\mathcal{H}}
\nwc{\calI}{\mathcal{I}}
\nwc{\calL}{{\mathcal L}}
\nwc{\calE}{\mathcal{E}}
\nwc{\calF}{\mathcal{F}}
\nwc{\calM}{\mathcal{M}}

\DeclareMathOperator{\dist}{dist}
\DeclareMathOperator{\diam}{diam}

\nwc{\one}{{\mathbbm{1}}}
\nwc{\eps}{\varepsilon}
\nwc{\inv}{^{-1}}
\nwc{\etal}{{\it et al.\ }}
\nwc{\id}{{\rm id}}
\nwc{\wkto}{\xrightharpoonup{\star}}

\newcommand{\vp}{\varphi}
\DeclareMathOperator{\ind}{\mathbbm{I}}
\DeclareMathOperator{\interior}{\rm int}
\nwc{\intr}[1]{{\kern0pt#1^\mathrm{o}}}%

\nwc{\ip}[2]{(#1)\cdot(#2) }
\nwc{\red}[1]{\textcolor{red}{[#1]}}
\nwc{\blue}[1]{\textcolor{blue}{#1}}
\nwc{\point}{\medskip $\bullet$\ }

\nwc{\calA}{{\mathcal A}}
\nwc{\calB}{{\mathcal B}}
\nwc{\calS}{{\mathcal S}}
\nwc{\calR}{{\mathcal R}}
\nwc{\calV}{{\mathcal V}}

\newcommand{\pts}{\psi_t^{*}}
\newcommand{\ptss}{\psi_t^{**}}

\newcommand{\MA}{Monge-Amp\`ere}
\newcommand{\mam}{\kappa}

\newcommand{\mdot}{\!\cdot\!}

\newtheorem{theorem}{Theorem}[section]
\newtheorem{lemma}[theorem]{Lemma}
\newtheorem{prop}[theorem]{Proposition}

\newtheorem{definition}[theorem]{Definition}

\makeatletter
\def\th@remark{%
  \thm@headfont{\bfseries}%
  \normalfont % body font
  \thm@preskip\topsep \divide\thm@preskip\tw@
  \thm@postskip\thm@preskip
}
\makeatother

\theoremstyle{remark}
\newtheorem{remark}[theorem]{Remark}
\numberwithin{equation}{section}

\begin{document}
\title{Rigidly breaking potential flows and a countable Alexandrov theorem for polytopes} 
\author[1]{Jian-Guo Liu\footnote{Email address: Jian-Guo.Liu@duke.edu} }
\author[2]{Robert L. Pego\footnote{Email address: rpego@cmu.edu} }
\affil[1]{Departments of Mathematics and Physics, 
Duke University, Durham, NC 27708, USA}
\affil[2]{Department of Mathematical Sciences,
Carnegie Mellon University, Pittsburgh, PA 15213, USA}

\maketitle
\begin{abstract}
We study all the ways that a given convex body in $d$ dimensions can break into countably many 
pieces that move away from each other rigidly at constant velocity, with no rotation or shearing.
The initial velocity field is locally constant a.e., but may be continuous and/or fail to be integrable.
For any choice of mass-velocity pairs for the pieces, such a motion 
can be generated by the gradient of a convex potential that is affine on each piece.  
We classify such potentials in terms of a countable version of a theorem of
Alexandrov for convex polytopes, and prove a stability theorem.
For bounded velocities, there is a bijection between the mass-velocity data
and optimal transport flows (Wasserstein geodesics) that are locally incompressible.

Given any rigidly breaking velocity field that is the gradient of a continuous
potential, the convexity of the potential is established under any of several 
conditions, such as the velocity field being continuous, the potential being semi-convex,
the mass measure generated by a convexified transport potential being absolutely continuous,
or there being a finite number of pieces. 
Also we describe a number of curious and paradoxical examples having fractal structure.
\end{abstract}


\pagebreak

\tableofcontents

\pagebreak

%s------------------------------------------
%s------------------------------------------
%s------------------------------------------
\section{Introduction}\label{s:intro}

Imagine that a brittle body, such as a crystal ball, shatters instantaneously
into pieces which fly apart from each other with constant velocities. 
Experience tells us to expect a large number of shards that
may be extremely small.

To model this in a simple way mathematically, we represent the body by a 
bounded convex open set $\Omega\subset\R^d$, and suppose its mass density is constant
and  normalized to unity.  We suppose that the body shatters into pieces represented by 
a countable collection of pairwise disjoint open subsets $A_i$ 
whose union $A = \bigsqcup_i A_i$ has full Lebesgue measure in $\Omega$.
For simplicity we presume the pieces travel by rigid translation with no rotation.
This means that any point $z$ in $A$ at time $t=0$ is transported to the point
\begin{equation}\label{d:Xt}
    X_t(z) = z + t v(z) 
\end{equation}
at time $t>0$, where the velocity field $v\colon\Omega\to\R^d$ is a constant $v_i$ on $A_i$.
It is natural to require the pieces to remain pairwise disjoint,
thus we require the transport map $X_t$ to be {\em injective} on $A$ for every $t>0$.
Given such a velocity field $v$, we will say that $v$ {\em rigidly breaks} 
$\Omega$ into $A_i$, $i=1,2,\ldots$.
The number of pieces $A_i$ may be finite or countably infinite.

We imagine that by observations around some time $t>0$ after shattering occurs,
we can determine the mass $m_i$ and the velocity $v_i$ for each piece. 
Our first result shows that these data suffice to determine all the pieces 
(and thus the entire flow)  in an essentially unique way,
{provided} we happen to know that the velocity is a {\em gradient of a convex potential}. 

Below, we call any function $\varphi\colon\Omega\to\R$ {\em locally affine a.e.} 
if it is affine on some neighborhood of $x$, for a.e.~$x\in\Omega$.
Given such a function we associate the set
\begin{equation}\label{d:Aset}
    A = \{x\in\Omega:\text{$\vp$ is affine on a neighborhood of $x$}\}.
\end{equation}
This is an open subset of $\Omega$ with full Lebesgue measure 
$\lambda(A)=\lambda(\Omega)$. 
The set $A$ has countably many components $A_i$, $i=1,2,3,\ldots$,
which are open and path-connected.  
For each $i$, $\vp$ is smooth on $A_i$ and its gradient $\vp$ is constant in a 
neighborhood of each point of $A_i$, so by path-connectedness
there must exist $v_i\in\R^d$ and $h_i\in\R$ such that 
\begin{equation}\label{e:affine}
 \varphi(z) = v_i\mdot z + h_i \quad\mbox{for all $z\in A_i$} \,.
\end{equation}
The following characterizes functions that are locally affine a.e.\ and convex.

\begin{theorem}\label{thm_countable}
Let $\Omega\subset\R^d$ be a bounded convex open set, let
$v_1,v_2,\ldots$ be distinct in $\R^d$, and let $m_1,m_2,\ldots$ be positive
so that $\sum_i m_i = \lambda(\Omega)$.
Then there is a function $\varphi$ on $\Omega$
(unique up to adding a constant)
that is locally affine a.e. and convex, 
so $\nabla\varphi=v_i$ on an open convex set $A_i$ with $\lambda(A_i)=m_i$.
\end{theorem}

Theorem~\ref{thm_countable} extends a geometric theorem of  Alexandrov on  unbounded 
convex polytopes~\cite{Alex1950} to the case of a countably infinite number of faces.  
Later in this introduction we will discuss this further. 

As a consequence of Theorem~\ref{thm_countable}, for any given mass-velocity data 
$m_i, v_i$, $i=1,2,\ldots$ as described, there {exists}
a velocity potential $\vp$ that is locally affine a.e.\ and convex and 
induces a partition of $\Omega$ as the data require.  Importantly, 
this map $X_t$ is injective on $A$ for all $t\ge0$, due to a simple lemma:

\begin{lemma}\label{lem:injective}
    Let  $\Omega\subset\R^d$ be open and convex, let 
    $\vp:\Omega\to\R$ be convex, and let $X_t(z)=z+t\nabla\vp(z)$ for all $z\in\Omega$.
    If $\vp$ is differentiable at $x,y\in\Omega$, then
    \begin{equation}
    |X_t(x)-X_t(y)|\ge |x-y| \quad\text{for all $t\ge0$.}
    \end{equation}
\end{lemma}

It is natural to wonder about a few things at this point. 
First, under what sort of conditions can we ensure that a rigidly breaking velocity field is
a gradient of a convex potential? Second, what is there to say about the
difference between having infinitely many pieces versus finitely many?
And further, is there a sense in which the flows depend continuously on the
mass-velocity data, justifying finite approximation? 
This paper is aimed at addressing these issues.

\smallskip
\noindent
{\bf Conditions for convexity.}
Our motivation for considering the first of these questions stems from 
our work \cite{LPS19} with Dejan Slep\v cev.
Certain results in that paper imply, roughly speaking, that any 
incompressible least-action mass transport flow 
must have initial velocity which is locally constant on an open set of full measure, 
equal to the gradient of  a potential  $\vp$ which is locally affine a.e.\ and {\em semi-convex}.  
Saying $\vp$ is semi-convex is equivalent  to saying that the function
 \begin{equation}\label{d:psit}
 \psi_t(z)=\frac12|z|^2+t\vp(z) 
 \end{equation}
is convex for some $t>0$.  In the immediate context, $\psi_t$ is the potential for
the transport map $X_t= \nabla\psi_t$, and the convexity of $\psi_t$ follows from Brenier's
theorem in optimal transport theory. 
(Below, we assume $\psi_t=+\infty$ outside $\bar\Omega$.)

In the present paper, we work in a somewhat more general situation.
We study flows produced by a.e.-locally affine potentials 
that start from a convex source domain
but need not have least action or even finite action. 
In this situation, a result of McCann~\cite{McCann97}, used to prove uniqueness 
of energy minimizers, directly implies that 
for any potential that is locally affine a.e.,  convexity is equivalent to semi-convexity.
From  \cite[Lemma~3.2]{McCann97} we immediately find the following.

\begin{theorem}\label{thm_semi}
   Let $\Omega\subset\R^d$ be a bounded open convex set, and assume
   $\vp\colon\Omega\to\R$ is locally affine a.e. 
   Then $\vp$ is convex if and only if it is semi-convex. 
\end{theorem}

Note that this result holds even without requiring the transport maps $X_t$ determined by
$v=\nabla\vp$ to be injective {\it a priori}. 
We can list three conditions, different from semi-convexity however, 
under which the injectivity suffices to entail the convexity of $\vp$
(and becomes equivalent to it, due to Lemma~\ref{lem:injective}).

\begin{theorem}\label{thm_main}
   Let $\Omega\subset\R^d$ be a bounded open convex set. Let $\vp\colon\Omega\to\R$
   be continuous and locally affine a.e., and define $A$ by \eqref{d:Aset}.
   Further, assume any one of the following:
   \begin{itemize}\setlength{\itemsep}{1pt}
       \item[(i)] The dimension $d=1$.
       \item[(ii)] The number of components of $A$ is finite.
       \item[(iii)] $\vp$ is $C^1$.
   \end{itemize}
   Then $\vp$ is convex if and only if 
   the map $z\mapsto X_t(z)=z+t\nabla\vp(z)$
   is injective on $A$  for all sufficiently small $t>0$. 
\end{theorem}
Under condition (i), the conclusion is easy to establish, of course. 
Condition (ii) and the local representation \eqref{e:affine}
together will imply that adjacent pieces must meet along flat faces where  
both convexity and injectivity reduce to a local monotonicity property for $\nabla\vp$.
For the case of condition (iii) we employ the Hopf-Lax formula  
which formally provides a solution to the initial-value problem
for a Hamilton-Jacobi equation with convex Hamiltonian, namely
\begin{equation}\label{e:HJ1}
    \D_t u + \frac12|\nabla u|^2 = 0, \qquad u(x,0)=\vp(x).
\end{equation}
The maps $X_t$ provide characteristics for this problem. 

Our last condition for convexity of $\vp$ is related to mass transport associated
with the convexification of $\psi_t$.
Below, we let $\psi_t^*$ denote the Legendre transform of $\psi_t$
for $t\ge0$, taking $\psi_t$ to be defined by~\eqref{d:psit} in the convex domain $\bar\Omega$ and
$+\infty$ outside. Then the convexification of $\psi_t$ is the double transform $\psi_t^{**}$.

\begin{theorem}
    \label{thm:monge_ampere}
   Let $\Omega\subset\R^d$ be a bounded open convex set. Let $\vp\colon\bar\Omega\to\R$
   be continuous and locally affine. Then $\vp$ is convex if and only if 
       for some $t>0$, the push-forward of Lebesgue measure
       under the (a.e.-defined) gradient of the convexification of $\psi_t$, written
       \[\mam_t = (\nabla\psi_t^{**})_\sharp \lambda\,,\]
   is absolutely continuous with respect to Lebesgue measure $\lambda$ on $\R^d$.
\end{theorem}

The proof of this theorem involves the 
second Hopf formula for solutions of the initial-value problem
for a different Hamilton-Jacobi equation which formally also has characteristics
given by $X_t$. Namely, for 
the following initial-value problem with convex initial data, 
\begin{equation}\label{e:HJ2}
    \D_t w + \vp(\nabla w) = 0, \qquad w(x,0) = \psi_0^*(x),
\end{equation}
with $\vp$ extended continuously to $\R^d$, the Legendre transform
$w=\psi_t^*$ is the unique viscosity 
solution of \eqref{e:HJ2}, 
according to a result of Bardi and Evans~\cite{BardiEvans}.

The push-forward measure  $\mam_t$ in Theorem~\ref{thm:monge_ampere}
is also described as the {\em Monge-Amp\'ere measure} determined by $\psi_t^*$,
as we discuss in Section~\ref{s:concentrations} below.
In space dimension $d=1$, the measure $\mam_t$ reduces to 
a mass measure induced by {\em sticky particle flow}, 
due to results of Brenier and Grenier~\cite{BrenierGrenier1998}.
When the velocity potential is non-convex, the velocity is not monotonically increasing, 
and the sticky particle flow is sure to form mass concentrations.
When the dimension $d>1$, our use of concentrations in $\kappa_t$ to 
characterize non-convexity for locally affine potentials $\vp$ 
is partly motivated by works of Brenier~\etal~\cite{brenier2003}
and Frisch~\etal~\cite{frisch2002reconstruction}
in the scientific literature.
These works describe links between a Monge-Amp\`ere equation, optimal transport, 
and mass density in the ``adhesion model'' in cosmology.  
The adhesion model is used to approximate the formation of mass-concentrating
structures in the universe such as sheets and filaments, 
see e.g.~\cite{weinberg1990largescale,vergassola1994burgers,gurbatov2012large}.

\begin{remark}\label{rem:conjecture}
It seems reasonable to conjecture that Theorem~\ref{thm_main} should
remain valid in general, without assuming {\em any} of the additional conditions (i)--(iii), 
only imposing some mild regularity assumption such as local Lipschitz regularity, perhaps. 
That is, non-convexity of $\vp$ should imply non-injectivity of $X_t$.
We have been unable to prove or disprove such a result. 
Thus it appears interesting to investigate various criteria under which 
injectivity suffices to ensure convexity.
Theorem~\ref{thm:monge_ampere} shows that non-convexity yields 
a measure-theoretic version of non-injectivity, however,
insofar as concentrations form instantaneously in $\mam_t$.
\end{remark}

\smallskip
\noindent
{\bf Incompressible least-action flows with convex source.}
Combined with our results from \cite{LPS19},
Theorems~\ref{thm_countable} and~\ref{thm_semi} provide a classification of 
action-minimizing mass-transport flows that are incompressible 
and transport Lebesgue measure in a given bounded open convex set $\Omega_0$ in $\R^d$
to Lebesgue measure in some other bounded open set. 
A precise description of such flows is provided in Theorem~\ref{thm_optimal} of Section~\ref{s:incompressible} below.
There we will show that they correspond in one-to-one fashion with countable sets 
$\{(m_i,v_i)\}$  of pairs consisting of positive masses $m_i$ and 
distinct velocities $v_i$ bounded in $\R^d$, such that $\sum_i m_i = \lambda(\Omega_0)$.

\bigskip
\noindent
{\bf Infinitely many vs.\ finitely many pieces.}
Characterizing convex and piecewise affine functions by volume and slope data 
relates to a classic geometric problem.
In 1897, Minkowski \cite{Minkowski1989,Alex1950}
proved that any compact convex polytope is uniquely determined, 
up to translation, by the {\em list of face normals and areas},
subject to a natural compatibility condition saying that the integral of 
the unit outward normal field over all faces must vanish.
Alexandrov  solved a version of this problem for unbounded convex polytopes
whose unbounded edges are parallel, and he presented his solution in his 1950 book
{\em Convex Polyhedra} \cite{Alex1950} (see sections 7.3.2 and 6.4.2). 
We quote Alexandrov's result essentially as reformulated by Gu \etal \cite{GuEtal16}
in terms of convex, piecewise affine functions, as follows. 

\begin{theorem}[Alexandrov]\label{thmAlex}
Let $\Omega$ be a compact convex polytope with nonempty interior in $\R^d$,
let $v_1,\ldots,v_k\in\R^d$ be distinct and let $m_1,\ldots,m_k>0$ so that
$\sum_{i=1}^k m_i = \lambda(\Omega)$. Then there is
convex, piecewise affine function $\varphi$ on $\Omega$ (unique up to adding a constant) 
so $\nabla\varphi = v_i$ on a convex set $A_i$ with volume $\lambda(A_i)=m_i$.
\end{theorem}
Alexandrov's unbounded polyhedra correspond to the supergraph
sets\[
\{(z,y)\in\R^d\times\R: z\in\Omega,\ y\ge \varphi(z)\},
\]
whose unbounded edges are parallel to the last coordinate axis.

We remark that in \cite{GuEtal16}, Gu \etal provided an elementary self-contained proof 
for a generalization of Theorem~\ref{thmAlex}, essentially equivalent here to minimizing
$\int_\Omega \varphi\, d\lambda$ as a function of the constants $h_i$ in the representation
\eqref{e:affine}
subject to the given volume constraints on $A_i$.
This is a variant of Minkowski's original proof (presented in \cite[sec. 7.2]{Alex1950})
of the existence of bounded polyhedra with prescribed face areas and normals
through a constrained maximization of volume. 
But this technique does not appear to work in the countably infinite case
of Theorem~\ref{thm_countable}.

In the case of finitely many pieces, in addition to
the conclusions stated in Alexandrov's theorem it is known that:
\begin{itemize}
\item[(i)] The velocity field $v=\nabla\varphi$ is discontinuous on $\Omega$ if $1<k<\infty$.
\item[(ii)] Each piece $A_i$ is the interior of a convex polytope.
\end{itemize}
Of course, property (i) is trivial since $\Omega$ is connected.
Property (ii) is due to the affineness from \eqref{e:affine} 
and the convexity of $\varphi$, which imply
$\varphi(z)\ge v_i\mdot z+h_i$ for all $z\in\Omega$. 
It follows $z\in A_i$ if and only if $z\in\Omega$ and 
\begin{equation}
v_i\mdot z + h_i > v_j\mdot z + h_j \quad\mbox{for all $j\ne i$}. 
\end{equation}
Equality is not possible since the $v_i$ are distinct and $A_i$ is open.
By consequence $A_i$ is the intersection of a finite number of half-spaces, i.e., a polytope.

In the case of infinitely many pieces, 
it turns out that neither (i) nor (ii) is necessarily true. 
A rigidly breaking velocity field can be continuous on $\Omega$, 
and a piece (shard) may assume any convex shape.
As the reader may suspect, examples involve fractal structure.
We will explore constructions involving 
Cantor sets, Vitali coverings, and Apollonian gaskets. 
Figure~\ref{f:spray} illustrates the latter: The shaded circles
indicate the sets $A_i+tx_i$, where the $A_i$ are Apollonian disks 
in the unit circle $\Omega$, $x_i$ is the center of $A_i$, and $t=0.5$.
See Section~\ref{ss:full_packing} for details.

Actually, continuity of the velocity is a highly paradoxical property, 
since it immediately implies that the flow images $X_t(\Omega)$ are connected, 
so seemingly not ``broken'' at all!  As we will show,
this phenomenon generates fat Cantor sets 
by ``expanding'' the standard Cantor set in a simple way.

\begin{figure}[!t]
\noindent\makebox[\textwidth]{%
\includegraphics[height=4.6in]{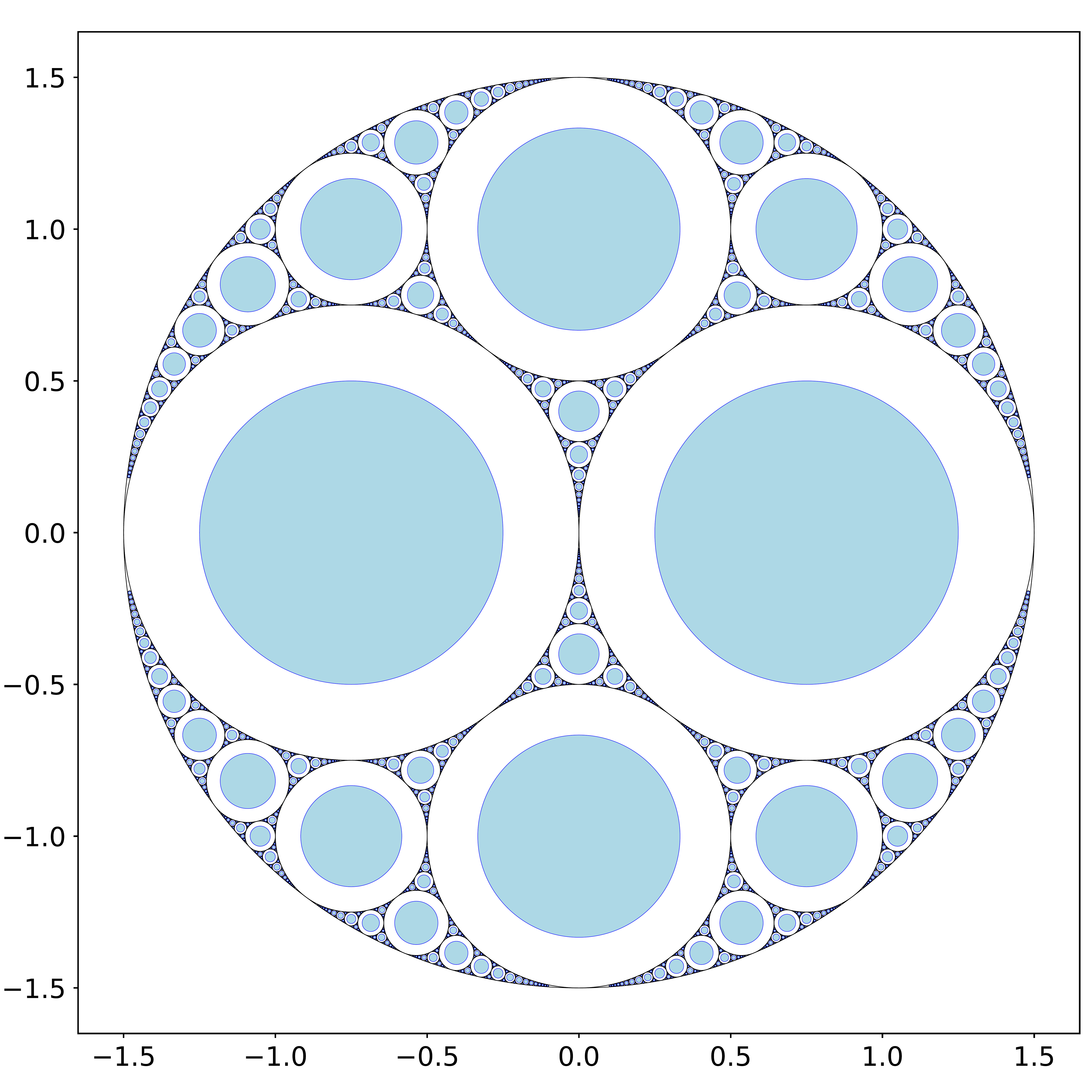}}
\caption{Breaking of an Apollonian gasket at $t=0.5$}  \label{f:spray} 
\end{figure}
\medskip
\noindent

{\bf Plan of the paper.}
Following this introduction, we first provide the proof of Theorem~\ref{thm_countable}
and Lemma~\ref{lem:injective} in Section~\ref{s:alex}.
In Section~\ref{s:1d} we study and classify rigidly breaking flows in the case of 
one space dimension, $d=1$. There we also discuss a paradoxical example 
with rigidly breaking but continuous velocity given by the Cantor function.

We complete the proof of Theorem~\ref{thm_main} 
in Sections~\ref{s:finite} and~\ref{s:C1}.
We handle case (ii)  in Section~\ref{s:finite}, where
we assume the flow rigidly breaks the convex domain into finitely many pieces. 
The case (iii), with $C^1$ potential, is handled in Section~\ref{s:C1},
making use of the Hopf-Lax formula for the solution of the Hamilton-Jacobi
equation~\eqref{e:HJ1}. 

We carry out the proof of Theorem~\ref{thm:monge_ampere} in Section~\ref{s:concentrations}.
In particular, in case
$\vp:\bar\Omega\to\R$ is continuous, locally affine a.e. and {\em non-convex},
Theorem~\ref{thm:MAt} shows that the \MA\ measure $\mam_t$ in Theorem~\ref{thm_main}
has a Lebesgue decomposition with a non-trivial singular part.

We next investigate the stability of rigidly breaking flows with respect to the 
mass-velocity data, in Section~\ref{s:stability}. There we show that weak-star convergence of
transported Lebesgue measure follows from weak-star convergence of pure point measures
naturally associated with the mass-velocity data.

In Section~\ref{s:incompressible} we complete our treatment of
incompressible least-action flows with convex source from \cite{LPS19},
establishing in Theorem~\ref{thm_optimal} that these flows are characterized
uniquely by their mass-velocity data $\{(m_i,v_i)\}$.

We study the possible shapes that the convex ``pieces'' $A_i$ may take
in Section~\ref{s:shape}.  In particular, we show that all the $A_i$ may be round
balls, corresponding to a full packing of $\Omega$ 
(e.g., any Apollonian or osculatory packing), and we show
that an individual component $A_i$ can assume any convex shape.

The paper concludes with a discussion that addresses three points.
We discuss how the continuity assumption on the potential $\vp$ in 
Theorem~\ref{thm_main} is ensured by the absence of shear (i.e., symmetry of
the distributional gradient $\nabla v$) and a local integrability condition.
We complete our Cantor-function example in Section~\ref{s:1d} showing
how fat Cantor sets are produced in a {uniformly} expanded way.
Finally, although we lack any characterization of rigidly breaking velocity fields
that are continuous when the dimension $d>1$, we discuss some constraints on such fields.

\section{Proof of a countable Alexandrov theorem}\label{s:alex}

Here we provide the proofs of Theorem~\ref{thm_countable} and Lemma~\ref{lem:injective}.
We prove Theorem~\ref{thm_countable} by a straightforward application
of a theorem of McCann~\cite{McCann95} which improved Brenier's theorem in optimal transport theory.

\begin{proof}[Proof of Theorem~\ref{thm_countable}]
Let the measure $\mu$ be given by $\lambda\mres\Omega$, Lebesgue measure restricted to 
the bounded convex open set $\Omega$, 
and let the measure
$\nu$ given as a combination of Dirac delta masses concentrated at the distinct points $v_i$, 
\begin{equation}\label{e:nu}
\nu = \sum_i m_i\delta_{v_i} , \quad\text{where $\sum_i m_i = \lambda(\Omega).$}
\end{equation}
With no moment assumptions, the main theorem in \cite{McCann95} produces a convex function
$\varphi$ on $\R^d$ whose gradient $T=\nabla\varphi$ is determined uniquely
a.e.\ in $\Omega$ and pushes $\mu$ forward to $\nu$.
The push-forward property $T_\sharp\mu = \nu$ has the consequence that 
the pre-image $\hat A_i$
of $\{v_i\}$ under the (a.e.-defined) gradient of $\vp$
is a Borel set $\hat A_i\subset\Omega$ with $\lambda(\hat A_i)=m_i$
and  $\nabla\varphi=v_i$ on $\hat A_i$.
Because $\Omega$ is connected, this determines $\vp$ up to a constant.

Since $\varphi$ is convex it is not difficult to deduce that $\varphi$ is affine 
on the closure of the convex hull of $\hat A_i$; 
see the lemma below.
Thus since $\lambda(\hat A_i)>0$, the closed convex hull has convex interior 
$A_i\subset \hat A_i\subset \bar A_i$ which is convex and has the same measure 
$\lambda(A_i)=\lambda(\bar A_i)=m_i$.
\end{proof}

\begin{lemma}\label{l:affine}
Assume $\Omega\subset\R^d$ is an open convex set and $f\colon\Omega\to\R$
is convex. 
(i) If $f$ is differentiable at points $x,y\in\Omega$ with $\nabla f(x)=\nabla f(y)$ 
then $ f$ is affine on the line segment connecting $x$ and $y$.
(ii) If $\nabla f$ is constant on a set $B$, then $ f$ is affine on 
the closed convex hull of $B$ in $\Omega$.
\end{lemma}

\begin{proof}
To prove (i), restrict $f$ to the line segment connecting $x$ to $y$,
defining $g(\tau)=f(x+\tau(y-x))$.
Then $g$ is 
differentiable at $\tau=0$ and $1$, with 
\[
g'(0)=
\nabla f(x)\mdot(y-x)=
\nabla f(y)\mdot(y-x)=
g'(1).
\]
Then $g$ is affine since it is convex. This proves (i), and we further note that
\begin{equation}\label{e:fdif}
f(x) - f(y) = \nabla f(y)\mdot (x-y).
\end{equation}
To prove (ii), by continuity it suffices to show $f$ is affine
on the convex hull of $B$. By Carath\'eodory's theorem on convex functions, 
each point in the convex hull is a convex combination of at most $d+1$ points
in $B$. Consider a convex combination $x = \sum_{j=1}^k t_jy_j$ with $y_j\in B$
and $t_j\ge0$ for all $j$ and $\sum t_j=1$. Invoking convexity and using 
\eqref{e:fdif}, we find that 
\begin{align*}
    \sum_{j=1}^k t_j f(y_j) &\ge f(x) \ge f(y_1)+\nabla f(y_1)\mdot(x-y_1)
    \\ & = \sum_{j=1}^k t_j \Bigl( f(y_1) + \nabla f(y_1)\mdot(y_j-y_1) \Bigr)
    = \sum_{j=1}^k t_j f(y_j).
\end{align*}
Hence $f$ is affine on the closed convex hull of $B$.
\end{proof}

\begin{remark}\label{r:pp}
Evidently, any arbitrary pure point measure $\nu$ on $\R^d$ having 
  total mass $\nu(\R^d)=\lambda(\Omega)$ can be expressed in the form \eqref{e:nu}
  for countable mass-velocity data that satisfy the assumptions of 
  Theorem~\ref{thm_countable}.  Reordering the data yield the same measure,
  hence there is a bijection between countable sets $\{(m_i,v_i)\}$ of such
  mass-velocity data and such pure point measures.
  McCann's main theorem from \cite{McCann95} associates a convex potential
  with any Radon measure $\nu$ on $\R^d$ having $\nu(\R^d)=\lambda(\Omega)$.
  The association of mass-velocity data with potentials in Theorem~1
  is obtained by restricting this to pure point measures.
  \end{remark}

  \begin{remark}
      In Section~\ref{s:stability} we will prove a stability (or continuity) theorem 
      for the flows $X_t=\id + t\nabla\vp$ determined by mass-velocity data as in the proof
      of Theorem~\ref{thm_countable} above. In Theorem~\ref{thm_stability} we show that
      for any sequence of pure point measures $\nu_n$ defined as in \eqref{e:nu},
      weak-star convergence of $\nu_n$ implies weak-star convergence of Lebesgue measure
      restricted to the transported sets $X^n_t(A^n)$ where $A^n$ is  the 
      open set defined as in \eqref{d:Aset} on  which $\vp_n$ is locally affine.
  \end{remark}

\begin{proof}[Proof of Lemma~\ref{lem:injective}]
  Let $\Omega\subset\R^d$ be open and convex, let $\vp\colon\Omega\to\R$ be convex,
  define $X_t(z) = z + t\nabla\vp(z)$ for $z\in\Omega$, and suppose $\vp$ is differentiable
  at two points $x,y\in\Omega$.  Convexity implies the graph of $\vp$ lies above 
  the tangent planes at $x$ and $y$, hence the well-known 
monotonicity condition follows:
  \begin{equation}\label{e:vp_monotone}
     (\nabla\vp(x)-\nabla\vp(y))\mdot(x-y)\ge 0\,. 
  \end{equation}
Thence
\[
\ip{X_t(x)-X_t(y)}{x-y} = |x-y|^2 + 
t\ip{\nabla\varphi(x)-\nabla\varphi(y)}{x-y} \ge |x-y|^2 ,
\]
and we infer $|X_t(x)-X_t(y)|\ge|x-y|$ by the Cauchy-Schwarz inequality.
\end{proof}

%s------------------------------------------
\section{One space dimension}\label{s:1d}

In order to develop understanding of 
rigidly breaking flows with a countably infinite number of components,
we consider the case of one space dimension.
We provide the easy proof of Theorem~\ref{thm_main} for this case,
and we illustrate and characterize the paradoxical possibility that a 
rigidly breaking velocity field may be continuous.

\subsection{Convexity in 1D}
\begin{proof}[Proof of Theorem~\ref{thm_main}(i)]
Make the assumptions of the theorem, including that (i) the dimension $d=1$. 
By Lemma~\ref{lem:injective} we know convexity of $\vp$ implies injectivity of $X_t$ on $A$ for all $t>0$.
Supposing that $X_t$ is injective on $A$ for all small enough $t>0$, 
we claim $\nabla\vp$ is necessarily increasing on $A$.
Each of the countably many components $A_i$ of the open set $A$ is an open interval.
Let $v_i$ be the constant value of $\nabla\vp$ on $A_i$.
The images $X_t(A_i)=A_i+tv_i$ then remain disjoint and preserve their initial order for all small $t>0$.
Let $A_i$, $A_j$ be any two component intervals of $A$ and assume
$A_i<A_j$, meaning $x<y$ whenever $x\in A_i$ and $y\in A_j$.
If $A_i$ and $A_j$ are adjacent, then clearly $v_i\le v_j$.  
If they are not adjacent, then the union of all intervals $A_k+tv_k$ with $A_i<A_k<A_j$ 
preserves its initial Lebesgue measure, 
hence the interval between $A_i+tv_i$ and $A_j+tv_j$ cannot shrink,
and so $v_i\le v_j$.
It follows $\vp$ is convex.  
\end{proof}

\subsection{Example: ``Cantor's elastic band''}\label{ss:cantor}

Take $\Omega=(0,1)\subset\R$, and consider the velocity field given by $v=c$ in $\Omega$,
where $c\colon[0,1]\to[0,1]$ is the standard {\em Cantor function}. 
The function $c$ is increasing yet continuous
on $[0,1]$ with $c(0)=0$ and $c(1)=1$,
and $c$ is locally constant on the open set $A=(0,1)\setminus\calC$, 
where $\calC$ denotes the standard Cantor set. 

For each component interval $A_i$ of $A$, let $v_i$ denote the value of $c$ on $A_i$.  
Then the flow in \eqref{d:Xt} is given by rigid transport in $A_i$, with
\[
X_t(z) = z + t v_i \,,  \quad z\in A_i. 
\]
Note that the distance between $X_t(A_i)$ and $X_t(A_j)$ increases linearly with $t$,
since $v_i<v_j$ for $A_i<A_j$. Thus $v$ rigidly breaks $\Omega$ into the $A_i$, according 
to our definition at the beginning of the introduction.

Indeed, the velocity potential $\varphi(z)=\int_0^z c(r)\,dr$ is convex and locally affine a.e. 
Yet $v=\nabla\vp$ is continuous. 
This seems paradoxical, for it implies the image $X_t(\Omega)$ remains {\em connected}
under the flow of the ``rigidly breaking'' velocity field $v$, 
and must comprise the full interval $(0,1+t)$!

Evidently, the injective maps $X_t$ ``stretch'' the interval $[0,1]$ to cover the longer interval $[0,1+t]$
by countably many rigidly translated images $X_t(A_i)$ together with the image of the Cantor set
$X_t(\calC)$.  
The union of the rigid images is the set $X_t(A)$, 
which is open and dense in $(0,1+t)$.
Of course the Lebesgue measure $\lambda((0,1+t))=1+t$, yet evidently 
\[
\lambda(X_t(A)) = \sum_i \lambda(X_t(A_i)) = \sum_i\lambda(A_i) = \lambda(A)=1 .
\]
What we infer from this is that the image $\calC_t:=X_t(\calC)$ is a {\em fat Cantor set}.
It is closed and nowhere dense in $(0,1+t)$, and has Lebesgue measure $\lambda(\calC_t)=t$.
The map $X_t$ has ``stretched'' the Cantor set $\calC$ with Lebesgue measure zero to a set with 
positive Lebesgue measure. 

In terms of physical intuition, we might fancifully imagine $\calC$ as consisting
of an ephemeral kind of matter having zero mass and always nowhere dense,
but infinitely stretchable so it can cover a set of positive Lebesgue measure.
The body $\Omega=(0,1)$ might be considered to model an elastic band 
made of a mixture of such stretchy stuff and ordinary rigid matter.
In this interpretation, deforming $\Omega$ to $X_t(\Omega)$ 
stretches the band but it does not disconnect it.

Less fancifully,  we wish to describe what is ``broken'' in a mathematically natural way.
For this we can focus on matter that has positive mass density. 
The rigid translation of the connected open pieces $A_i$ induces a mass measure 
$\nu_t$ on the image domain $X_t(\Omega)$ that is 
{\em not} the restriction of Lebesgue measure to $X_t(\Omega)$.
Instead, $\nu_t$ is the restriction of Lebesgue measure to the 
disconnected open (yet dense) set $X_t(A)= \bigsqcup_i X_t(A_i)$. 
We can say the body $\Omega$ is broken into the disconnected components $X_t(A_i)$
that carry all the mass.
This induced mass measure $\nu_t$ is nothing but the push-forward 
under $X_t$ of $\lambda\mres\Omega$, Lebesgue measure restricted to $\Omega$.
We have  $(X_t)_\sharp(\lambda\mres\Omega) = \lambda\mres X_t(A)$
in the present example,  and this {\em differs} from $\lambda\mres X_t(\Omega)$.
While one can make different choices of the set $A$ with this property,
it seems natural to take $A$ to be the 
open set in \eqref{d:Aset} on which the
velocity potential is locally affine. 

In Fig.~\ref{f:cantor} we illustrate this example by plotting the 
velocity $v=c$ as a function of transported position $x=X_t(z)=z+tc(z)$.
The transported pieces $X_t(A_i)$ are (non-singleton) level sets of
the transported velocity $v=f(x,t)$,
which is constant along the flow lines $x=z+tc(z)$. 
As a side remark, it is interesting to note that 
while the partial derivative 
$\D f/\D x = 0$ in every translated component $X_t(A_i)$,
it turns out that $\D f/\D x = 1/t$ a.e.~in the 
fat Cantor set $\calC_t$, meaning these sets expand uniformly in time. 
We defer proof to the Discussion below, see Proposition~\ref{prop:Lax}.

\begin{figure}\begin{center}
\includegraphics[width=4.2in]
{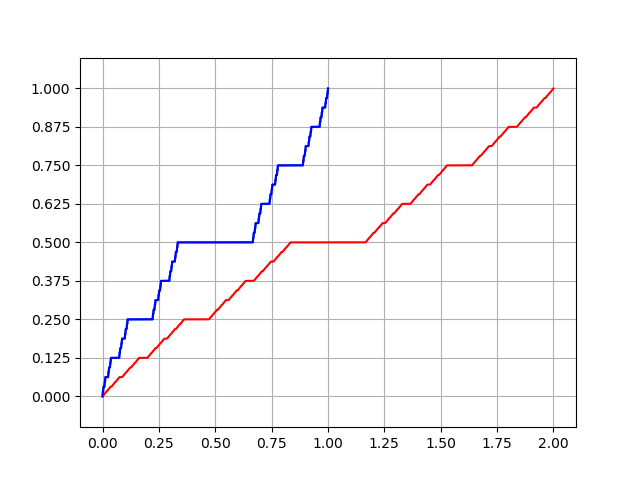}
\put(-310, 165){\Large $v$}
\put(-20, 35){\Large $x$}
\caption{Cantor expansion wave: 
$v=c(z)$ {\it vs.}~$x=z+tc(z)$ at $t=0$ and $1$.}\label{f:cantor}
\end{center}
\end{figure}

\subsection{Characterization of continuity in one dimension}

The Cantor-function example generalizes to provide necessary and sufficient
conditions for a rigidly breaking velocity field to be continuous
when $d=1$.
Recall that by Theorem~\ref{thm_main}(i), such a velocity field must
be the derivative of a $C^1$ potential $\varphi$ 
that is convex and locally affine a.e.

\begin{prop}\label{p:cacx1d}
Let $\Omega\subset\R$ be a bounded open interval, and let $\varphi$
be convex and  locally affine a.e.~on $\Omega$, with $\varphi'$ 
taking the distinct values $\{v_i\}$ on an open set of full measure in $\Omega$.
Then $\varphi$ is $C^1$ if and only if the sequence $\{v_i\}$ is dense in an interval.
\end{prop}

\begin{proof}
Suppose $\varphi$ is convex and locally affine a.e.,
so $\varphi'$ is defined and constant on each component of an open set $A$ 
of full measure in $\Omega$.  If $\varphi$ is $C^1$, then 
the continuous image $\varphi'(\Omega)$ must be connected, hence
an interval, and $\varphi'(A)=\{v_i\}$ must be dense in it. 
On the other hand, if $\varphi'(A)$ is dense in an interval $I$,
then because $\varphi'$ is increasing on $A$, we have $\varphi=\int v\,dx$
where the function given by 
\[
v(x) = \lim_{z\uparrow x,\, z\in A} \varphi'(z) \,,\quad x\in\Omega,
\]
is increasing with no jump discontinuities. So $v$ is continuous,
and $\varphi$ is $C^1$.
\end{proof}

\begin{remark} By Theorem~\ref{thm_countable},
for any sequence $\{v_i\}$ of distinct values dense in an interval,
such $C^1$ potentials exist and are specified uniquely by any positive sequence $\{m_i\}$ with $\sum_i m_i=\lambda(\Omega)$. 
In this case $v=\varphi'$ is a Cantor-like function, 
continuous and increasing on $\Omega$ and constant on an interval $A_i$ with 
$\lambda(A_i)=m_i$. 
\end{remark}

%\pagebreak
%s------------------------------------------
\section{Finitely many pieces}\label{s:finite}
In this section we prove Theorem~\ref{thm_main} under condition (ii)
which states that the number of components $A_i$ of $A$ is finite. 
Recall that convexity of $\vp$ implies injectivity of $X_t$ by Lemma~\ref{lem:injective}.
Briefly, our strategy for proving the converse
will be to show that if $\vp$ is non-convex, then two adjacent components
must have velocities that force their images under the flow
$X_t$ to overlap immediately for $t>0$. We do this by finding a line segment
along which the restriction of $\vp$ is non-convex and intersects
$\D A$ only at finitely many points on flat ``faces'' between adjacent components.

Throughout this section we work under the basic assumptions of Theorem~\ref{thm_main},
and assume the dimension $d>1$.
Recall we assume $A$ is given by \eqref{d:Aset} 
and its components $A_i$ are open and connected and their number $N$ is finite. 
The case $N=1$ is trivial, so assume $N>1$.   
Given that $\vp$ is locally affine on $A$ and continuous on $\Omega$,
there exist $v_1,\ldots,v_N\in \R^d$ and $h_1,\ldots,h_N$ such that 
the representation \eqref{e:affine} extends by continuity to say
\begin{equation}\label{e:vpi}
    \vp(z) = v_i\mdot z + h_i \,, \quad z\in \bar A_i, 
    \quad i\in [N]=\{1,\ldots,N\}.
\end{equation}
Repeated values are possible.  
By  \eqref{e:vpi} and  \eqref{d:Aset},
each point in the interior of $\bar A_i$ must be in $A$.
Since $\bar A_i$ is disjoint from $A_j$ for $j\ne i$, 
the interior of $\bar A_i$ is $A_i$ 
and $\bar A_i$ is the disjoint union of $A_i$ and $\D A_i$.

\subsection{Geometry of the pieces}

We begin by precisely describing some of the geometric structure of 
the dense open set $A$ and its boundary (or complement) in $\Omega$.   
Define an ``adjacency function'' by
\begin{equation}\label{d:Jz}
\calI(z)=\{i\in[N]: z\in \bar A_i\}
\quad\text{for each $z\in\Omega$}.
\end{equation}
Evidently the cardinality $\# \calI(z)=1$ if 
$z\in A$. Define ``face'' and ``edge'' sets respectively by
\begin{equation}
    \label{d:Omega123}
    F =\{z\in\Omega: \# \calI(z) = 2\}, \quad
    E =\{z\in\Omega: \# \calI(z) \ge 3\}. 
\end{equation}
\begin{lemma}
Make the assumptions of Theorem~\ref{thm_main} including condition (ii). 
Let $A^c=\Omega\setminus A$. Then 
$A^c= \D A\cap\Omega$ and we have
 \[
 A = \{z\in\Omega: \#\calI(z)=1\}\,, \quad
 A^c = \{z\in\Omega:\#\calI(z)\ge2\} = F\cup E.
 \]
\end{lemma}
\begin{proof}
Because $A$ is open and dense, $\D A\cap \Omega= A^c$.
The finite union $\bigcup_i \bar A_i$ is closed and contains $A$, 
hence $\bar A=\bar\Omega$, so necessarily $\#\calI(z)\ge1$ for all $z\in\Omega$.

Now, let $z\in A^c$. It remains to show $\#\calI(z)\ge2$. 
Fix $i$ with $z\in \bar A_i$. Necessarily $z\in\D\bar A_i$
since $z\notin A_i$. For any $k>0$ there exists $y_k\in \Omega\setminus \bar A_i$
with $|y_k-z|<1/k$. Then since $N$ is finite, some subsequence
of the $y_k$ lie in $\bar A_j$ for some fixed $j\ne i$.
It follows $z\in \bar A_j$, hence $\#\calI(z)\ge2$ as required.
\end{proof}

Next, for all $i,j\in[N]$ with $i\ne j$ we define 
\begin{equation}\label{d:Hij}
   H_{ij}  = \{z\in\R^d:  v_i\mdot z+h_i = v_j\mdot z+h_j\}.
\end{equation}
Provided $v_i\ne v_j$ this set is a hyperplane of codimension 1.
Let $\calH$ denote the collection of these co-dimension--1 sets.

\begin{prop}\label{prop:finite}
Make the assumptions of Theorem~\ref{thm_main} including condition (ii). 
Then:
\begin{itemize} 
    \item[(a)] The set $A^c$
   is contained in a finite union of codimension--1 hyperplanes.
\item[(b)] For any $z\in A^c$, $z\in F$ if and only if  
$z$ lies in $H_{ij}\cap B$ for some hyperplane $H_{ij}$ in $\calH$
and some open ball $B$ with $B\subset A_i\cup A_j\cup H_{ij}$. 
\item[(c)] The set $E$ is contained in a finite union of codimension--2 hyperplanes.
\end{itemize}
\end{prop}
\begin{proof}
    (a) Let $z\in A^c$. Then $\# \calI(z)\ge 2$. 
For each pair of indices $i,j\in \calI(z)$, we must have 
\begin{equation}\label{e:ijeq}
   v_i\mdot z + h_i = v_j\mdot z + h_j  \,.
\end{equation}
Some such pair exists with $v_i\ne v_j$, 
for $\vp$ is not affine in any neighborhood of $z$ since $z\notin A$. 
Then $z$ lies in the codimension--1 hyperplane $H_{ij}$. This proves (a).

(b) Let $z\in A^c$ and assume $z\in F$. Then $\calI(z)=\{i,j\}$ with $v_i\ne v_j$, so $z\in H_{ij}$,
and the distance from $z$ to $\bar A_k$ is positive for any $k\notin \calI(z)$. 
Since $A$ has only finitely many components by (ii), 
there are only finitely many such $k$. Then for any small enough open ball $B$ containing $z$, 
$B\subset \bar A_i\cup\bar A_j$, while both 
$B\cap\D A_i$ and  $B\cap\D A_j$ lie in $H_{ij}$. Hence 
$B\subset A_i\cup A_j\cup H_{ij}$.

Conversely, suppose $z\in A^c$ and $z\in H_{ij}\cap B$ for 
some hyperplane $H_{ij}$ in the finite collection $\calH$
and some open ball $B\subset A_i\cup A_j\cup H_{ij}$.
Then for each $k\in\calI(z)$, $B\cap A_k$ is a non-empty open set.
Whenever $k\notin\{i,j\}$, however, since $A_k\cap A_i$ and 
$A_k\cap A_j$ are empty, necessarily $B\cap A_k = B\cap H_{ij}$.
This set must be empty since it is open and 
$H_{ij}$ has co-dimension 1.  It follows $\calI(z)=\{i,j\}$
since $\#\calI(z)\ge2$. Hence $z\in F$.

(c) If $z\in E$, then $z\in A^c$ but $z\notin F$. It follows from part (a) 
that $z$ must lie in some hyperplane $H_{ij}$ of $\calH$, and from part (b) that $B\setminus H_{ij}$
intersects $A^c$ for every sufficiently small ball $B$. 
Then since $\calH$ is finite, necessarily $z$ must lie in the intersection of two different 
(i.e., non-coinciding) hyperplanes of $\calH$. 
Such intersections form a finite collection of hyperplanes of co-dimension 2.
\end{proof}

\subsection{Convexity for finitely many pieces}

If the transport map $X_t(z)=z+t\nabla\vp(z)$ is injective on $A$ for small $t>0$, 
Proposition~\ref{prop:finite} allows us to prove the following local monotonicity property.
\begin{lemma}\label{lem:local_mono}
    Assume $X_t$ is injective on $A$ for all sufficiently small $t>0$. 
    Suppose $\bar A_i\cap \bar A_j$ contains a point $z\in F$. 
    Then in any sufficiently small open ball containing $z$, 
    \[
    (\nabla\vp(x)-\nabla\vp(y))\mdot(x-y) >0
    \quad\text{for all  $x\in A_i$ and $y\in A_j$.}
    \]
\end{lemma}
\begin{proof}
   Necessarily $\calI(z)=\{i,j\}$ and $z\in H_{ij}$. Let $B$ be an open ball as given by Proposition~\ref{prop:finite}(b).
   Let $u$ be a unit vector orthogonal to the hyperplane  $H_{ij}$ pointing from $A_j$ toward $A_i$. 
   By the definition of $\calH$, $v_i\ne v_j$ and  $v_i-v_j=au$ for some nonzero $a\in\R$. 
   For all small enough $b>0$, $z_i:=z+bu\in A_i$ and $z_j:=z-bu\in A_j$.
   The injectivity hypothesis on $X_t$ implies
   \[
   0\ne X_t(z_i)-X_t(z_j) = z_i - z_j + t(v_i-v_j) = (2b + ta)u \,,
   \]
   for all sufficiently small positive $b$ and $t$. 
   This necessitates $a>0$, and implies $(v_i-v_j)\mdot(z_i-z_j) = 2ab>0$.
   This entails the result, since both $u\mdot (z_i-z_j)$ and $u\mdot(x-y)$ are positive
   for $x,y\in B$ with $x\in A_i$, $y\in A_j$.
\end{proof}

Now we are able to complete the proof of Theorem~\ref{thm_main} under condition (ii).

\begin{proof}[Proof of Theorem~\ref{thm_main}(ii)]
1. Assume $X_t$ is injective on $A$ for all sufficiently small $t>0$,
but $\vp$ is not convex. Then there must 
exist distinct $x,y\in\Omega$ and $\hat\tau\in(0,1)$ such that 
\begin{equation}\label{e:non_tau}
\vp(x\hat\tau + y(1-\hat\tau)) > \vp(x)\hat\tau +\vp(y)(1-\hat\tau).
\end{equation}
We may take $x,y\in A$, since $\vp$ is continuous and $A$ is dense. 
Let $u=x-y$ and let $u^\perp$ be the hyperplane of co-dimension 1 through the origin and orthogonal to $u$. 
The orthogonal projection $P_u$ of $\R^d$ onto $u^\perp$ maps the line segment
$\overline{xy}$ to a point, where
\[
\overline{xy} = \{ x\tau + y(1-\tau): \tau\in[0,1]\} \,.
\]
The same projection maps the set $E$ of Proposition~\ref{prop:finite} 
into a finite union of hyperplanes of relative codimension 1 in $u^\perp$. 
The same is true for any co-dimension--1 hyperplanes $H_{ij}$ in $\calH$ that happen
to have $u$ in their tangent space.
There exist arbitrarily small $v\in u^\perp$
such that $P_ux + v = P_u(x+v)$  does not lie on any of these projected hyperplanes.
Since $P_u(x+v)=P_u(y+v)$, we may then replace $x,y$ by $x+v, y+v$ 
and ensure that the line $\overline{xy}$ is disjoint from $E$ and transverse
to every hyperplane $H_{ij}\in\calH$ that it intersects, 
and \eqref{e:non_tau} still holds.
The line $\overline{xy}$ then intersects $A^c$
only at points of $F$, and only at finitely many of those.  As the line $\overline{xy}$
cannot be contained in a single component of $A$, at least one such intersection point exists.

2.  The function  $\hat \vp(\tau)= \vp(x\tau+y(1-\tau))$ defined for $\tau\in[0,1]$ satisfies
\[
\frac{d\hat\vp}{d\tau} = \nabla \vp(x\tau+y(1-\tau)) \mdot (x-y)
\]
whenever $x\tau+y(1-\tau)\in A$. Then $d\hat\vp/d\tau$ is locally constant on $(0,1)$,
with a jump at any value of $\tau$ where  $z=x\tau+y(1-\tau)\in A^c$. Necessarily $z\in F$
by step 1, and by applying Lemma~\ref{lem:local_mono} we can conclude that 
$d\hat\vp/d\tau$
makes a {\em positive} jump at such a value of $\tau$. This implies $\hat\vp$ is convex
on $(0,1)$, contradicting \eqref{e:non_tau}.  Hence $\vp$ is convex in $\Omega$.
This finishes the proof of Theorem~\ref{thm_main} under condition (ii).
\end{proof}

%\pagebreak
%s------------------------------------------
\section{Continuously differentiable potentials}\label{s:C1}

In order to prove Theorem~\ref{thm_main} under condition (iii), it suffices to prove 
the following proposition.  
The proof is motivated by the idea that the transport maps $X_t$
are related to characteristic curves for the Hamilton-Jacobi initial-value problem
\[
\D_t u + \tfrac12|\nabla u|^2 = 0, \qquad u(x,0)=\vp(x),
\]
whose solution, under suitable conditions, is given by the Hopf-Lax formula
\begin{equation}\label{e:HopfLax2}
u(x,t) = \min_y \left( \frac{|x-y|^2}{2t} + \vp(y) \right).
\end{equation}
The proof will make use of Theorem~\ref{thm_semi} in order to ensure  that a 
certain needed minimizer exists inside $\Omega$.

\begin{prop}
Let $\Omega$ be a bounded open convex set in $\R^d$. Let
$\vp\colon\Omega\to\R$ be $C^1$ on $\Omega$ and locally affine a.e.
Let $A$ be the open set in \eqref{d:Aset}.
Suppose  $\vp$ is not convex.
Then  $X_t$ is non-injective on $A$ for all sufficiently small $t>0$.
\end{prop}

\begin{proof}
1.  Suppose $\vp$ is not convex. Then it is not convex in some nonempty subset
\[
\Omega_\eps = \{x\in\Omega: \dist(x,\D\Omega)>\eps\}, 
\]
for some $\eps>0$ (fixed). 
The set $\Omega_\eps$ is convex itself, as is easily shown.
Let $L=\sup_{\overline{\Omega_{\eps}}}|\nabla\vp|$ and 
$M=\sup_{\overline{\Omega_{\eps/4}}}|\vp|$.  Fix $t>0$ so $Lt<\eps/2$ and $Mt<\eps^2/64$.

2.  By Theorem~\ref{thm_semi},  $\vp$  is not  semi-convex on $\Omega_\eps$, 
hence $\psi_t(z)=\frac12|z|^2+t\vp(z)$ 
is not convex on $\Omega_\eps$, and cannot coincide with its convexification $\psi_t^{**}$. 
Since $A$ is dense in $\Omega_\eps$, there exists $z_0\in\Omega_\eps\cap A$ such that $\psi_t(z_0)>\psi_t^{**}(z_0)$.
Then $v_0=\nabla\vp(z)$ is constant for all $z$ in some small neighborhood of  $z_0$ contained in $A$.
Let $x = z_0 + t v_0$.
Then $|x-z_0|\le tL<\eps/2$, so $x\in \Omega_{\eps/2}$.

3. Taking the min over $y\in\overline{\Omega_{\eps/4}}$  in the Hopf-Lax formula \eqref{e:HopfLax2},   
we have $u(x,t)\le M$ by taking $y=x$. When $y\in \D\Omega_{\eps/4}$ we have 
$|x-y|>\eps/4$, whence
\[
\frac{|x-y|^2}{2t} + \vp(y) \ge \frac{\eps^2}{32 t} - M >M.
\]
Hence any minimizer $y_1$ in $\overline{\Omega_{\eps/4}}$ lies in the open set $\Omega_{\eps/4}$, and it follows 
\[
x=y_1+t\nabla\vp(y_1) = z_0+t\nabla\vp(z_0),
\quad \text{i.e.,}\quad X_t(y_1)=X_t(z_0).
\]
Moreover, with $h=tu(x,t)-\frac12|x|^2$,
we have $h+x\mdot y\le \psi_t(y)$ for all $y\in\Omega_{\eps/4}$, with equality at $y_1$.
Since the affine function  $h+x\mdot y\le\psi_t^{**}(y)\le \psi_t(y)$ for all $y$,
we infer $\psi_t(y_1)=\psi_t^{**}(y_1)$.
Hence $y_1\ne z_0$. 

4.  Note $z_0 = x-tv_0=X_t(y_1)-tv_0$. 
Because $A$ is dense and $X_t$ is continuous, 
we can find $\tilde y_1\in A$ such that 
$\tilde z_0 := X_t(\tilde y_1)-tv_0 \in A$  with $\tilde z_0\ne \tilde y_1$.
Yet $v_0=\nabla\vp(\tilde z_0)$ and $\tilde x := X_t(\tilde z_0)=X_t(\tilde y_1)$.
This contradicts the assumed injectivity of $X_t$ on $A$, and concludes the proof.
\end{proof}

\begin{remark}
This Proposition handles locally affine functions $\vp$ that resemble
the Cantor expansion example in 1D in that they have continuous gradient. 
\end{remark}

\begin{remark}
   We suspect that if $\vp$ is $C^1$, locally affine and non-convex then $X_t$ is non-injective for {\em every} $t>0$. 
   But we leave this issue aside for the present.
\end{remark}

%s------------------------------------------
\section{Mass concentrations in convexified transport}\label{s:concentrations}

The main goal of this section is to prove Theorem~\ref{thm:monge_ampere}.
As mentioned in the Introduction, the measure $\mam_t$ is related to the second Hopf formula
for the solution to the following initial value problem with convex initial data:
\begin{equation}\label{e:Ham2}
\D_tw + \vp(\nabla w)=0, \qquad w(x,0)=\psi_0^*\,.
\end{equation}
Here 
$f^*(x)=\sup_{z\in\R^d} x\mdot z - f(z)$ 
denotes the Legendre transform of $f$, and
\begin{equation}
\psi_0(y) = \tfrac12 |y|^2 + \ind_{\bar\Omega}(y),
\end{equation}
where $\ind_S$ is the indicator function of the set $S$: 
$\ind_S(z)=0$ if $z\in S$, $+\infty$ if $z\notin S$.
Since $\psi_0^*(x)=\sup_{z\in\bar\Omega}x\mdot z-\frac12|z|^2$ and this is Lipschitz, 
results of Bardi and Evans~\cite{BardiEvans} imply that 
(regarding $\vp$ as extended continuously to all of $\R^d$)
the unique viscosity solution of \eqref{e:Ham2} is given by 
the second Hopf formula, which states
\begin{equation}\label{e:2ndHopf}
w_t = \psi_t^*\qquad\text{where}\quad\psi_t= \psi_0+t\vp \,. 
\end{equation}

We will make no direct use of this fact. 
Instead, we will focus attention on what is known as the {\em \MA} measure
for the convex function $\pts$.  This is the Borel measure 
whose value on each Borel set $B$ in $\R^d$ is given by
\begin{equation}\label{d:mut}
\mam_t(B) = \lambda(\D\pts(B)) =  \lambda\left(\bigcup_{x\in B}\D \pts(x)\right).
\end{equation}
See \cite[p.~7]{Figalli-MongeAmpere}.  
Results to be quoted below show that this agrees with 
the pushforward formula for $\mam_t$  stated in Theorem~\ref{thm:monge_ampere}.
\begin{remark}
    For fixed $t$, the fact that the function $w_t$ has \MA\ measure given by 
    $\mam_t$
    simply means that $u=w_t$ is the {\em Alexandrov solution} to the \MA\ equation 
    \[
    \det D^2 u =  \mam_t\,. 
    \]
\end{remark}

\subsection{Convex mass transport}

First, we establish that absolute continuity of $\mam_t$ is a necessary consequence
when $\vp$ is convex. Indeed, $\mam_t$ is given by locally rigid transport in this case.

\begin{prop}\label{prop:mu_convex}
Let $\Omega$ be a bounded open convex set in $\R^d$, and let
$\vp\colon\bar\Omega\to\R$ be continuous. 
Assume $\vp$ is locally affine a.e.,
and let $A$ be the open set defined in \eqref{d:Aset}.
Further assume $\vp$ is convex.
Then for all $t>0$, the \MA\ measure in \eqref{d:mut} is given by
\[\mam_t=\lambda\mres X_t(A)\,,\]
Lebesgue measure on the set $X_t(A)$ whose each component 
is translated rigidly.
\end{prop}

\begin{proof}
    To begin we note that for each $t>0$, $\psi_t:\R^d\to(-\infty,\infty]$ is convex,
    lower semicontinuous, and finite on $\bar\Omega$. 
    Then $\psi_t=\ptss$  by the Fenchel-Moreau theorem; see \cite[\S1.4]{Brezis2011}.
    Several further basic facts regarding the subgradients $\D\psi_t$ in this context are the following
    (see \cite[Appendix A]{LPS19} for simple proofs of (2) and (3)):
\begin{itemize}
\item[(1)] 
The inverse $(\D\psi_t)\inv=\D\psi_t^*$, according to Rockafellar~\cite[Thm 23.5]{Rockafellar}.
\item[(2)] 
$x\in \D\psi_t(y)$ iff $x=y+z$ with $z\in\D(\ind_{\bar\Omega}+t\vp)(y)$.
\item[(3)] $\D\psi_t$ has range $\R^d$.  
\end{itemize}

Let $B$ be a Borel set in $\R^d$ and 
let $x\in B$, $y\in(\D\pts)(x)$. 
Then $x\in\D\psi_t(y)$ by (1), 
whence necessarily $y\in \bar\Omega$, for otherwise $\D\psi_t(y)$ is empty.
As the set $\bar\Omega\setminus A$ has Lebesgue measure zero,
by (1) it follows
\[
\mam_t(B) 
= \lambda\left(A\cap(\D\psi_t)\inv(B)\right).
\]

Let the components of $A$ be denoted $A_i$ and let $v_i$ be the value of $v=\nabla\vp$ in $A_i$.
We claim that for each $i$,
\[
A_i\cap (\D\psi_t)\inv(B) 
= A_i\cap (B-tv_i).
\]
Indeed, if $y\in A_i\cap (\D\psi_t)\inv(B)$, there exists $x\in B$ with 
$x\in \D\psi_t(y)=\{y+tv_i\}$ so $y\in B-tv_i$. And if $y\in A_i\cap(B-tv_i)$,
$x:=y+tv_i=\nabla\psi_t(y)\in B$ so $y\in (\D\psi_t)\inv(B)$.

Recalling that $X_t(y)=y+t\nabla\vp(y)$ is injective on $A$ by 
Lemma~\ref{lem:injective},
by translation invariance of Lebesgue measure it follows
\[
\mam_t(B) = \sum_i \lambda(A_i\cap(B-tv_i)) = \sum_i\lambda(X_t(A_i)\cap B) =
\lambda(X_t(A)\cap B).
\]
Hence $\mam_t=\lambda\mres{X_t(A)}$. 
\end{proof}

\begin{remark} The situation in Proposition~\ref{prop:mu_convex} 
provides a particularly simple special solution of a \MA\ equation 
of the general form
\begin{equation}\label{e:MA-McCann}
\rho'(\nabla\psi(y)) \det D^2\psi(y) = \rho(y) \,,
\end{equation}
in which $\psi = \psi_t$, $\rho = \one_\Omega$ and $\rho' = \one_{X_t(A)}$.
In~\cite[Sec.~4]{McCann97}, McCann proved that for any convex $\psi$,
if $\rho$ is the density of an absolutely continuous probability measure
also denoted $\rho$ in  the interior of the domain of $\psi$, 
and if $\rho'=\nabla\psi_\sharp \rho$
is absolutely continuous with density also denoted $\rho'$,
then the \MA\ equation \eqref{e:MA-McCann} holds a.e.~in $\Omega$,
where the Hessian $D^2\psi(y)$ is interpreted in the Alexandrov sense.
\end{remark}

%\vfil \pagebreak
\subsection{Non-convex mass transport}

Our next goal is to  associate non-convexity of $\vp$ with 
formation of singular concentrations in $\mam_t$, as follows.
Recall we assume $\vp$ is continuous on $\bar\Omega$, and
$\psi_t$ is given by \eqref{e:2ndHopf}, taking the value $+\infty$
outside the convex set $\bar\Omega$. 
Then $\psi_t$ is lower semicontinuous on $\R^d$ and its convexification
is $\ptss$, which is also finite only on $\bar\Omega$.
The Legendre transform of $\ptss$ is $\psi_t^{***}=\pts$ by 
the Fenchel-Moreau theorem, and by
\cite[Thm.~23.5]{Rockafellar} cited in (1) above we have the inverse relation 
\[
(\D\ptss)\inv = \D\pts \,.
\]
Hence the \MA\ measure $\mam_t = (\D\ptss)_\sharp(\lambda\mres \Omega)$,
for this simply means 
\begin{equation}\label{e:ktpullback}
\mam_t(B)=\lambda((\D\ptss)\inv(B)),
\end{equation}
which is the same as \eqref{d:mut}.
This is not different from the formula 
in Theorem~\ref{thm:monge_ampere}, saying
$\mam_t = (\nabla\ptss)_\sharp \lambda$,
because any set $(\D\ptss)\inv(B)$ is contained in $\bar\Omega$ and
can differ from $(\nabla\ptss)\inv(B)$ only 
at points where $\ptss$ is not differentiable, which form a Lebesgue null set.
(A similar point is made in \cite[Lemma~4.1]{McCann97} in a more general context.)

Our main result in this section is the following theorem
which completes the proof of Theorem~\ref{thm:monge_ampere}
by establishing the sufficiency of the absolute continuity of $\mam_t$ 
for the convexity of $\vp$.
It shows that when $\vp$ is non-convex, the mass evolution determined by the Monge-Amp\`ere measure
$\mam_t$ decomposes into a part 
given by rigid translation $z\mapsto\nabla\psi_t(z)=z+t\nabla\vp(z)$ locally, 
and a nontrivial remainder that instantaneously concentrates on a null set. 
We comment on the relationship of this result with the adhesion model 
of cosmology at the end of this section.

\begin{definition}\label{d:Thetat}
For each $t>0$ we define the ``touching set''
\begin{align}
\Theta_t &= 
\{y\in\bar\Omega:  \psi_t(y)=\ptss(y)\}\,,
\label{d:Ths}
\end{align}
and for $t=0$ we define $\Theta_0=\bar\Omega$.
We let $\intr{\Theta_t}$ denote the interior of $\Theta_t$.
\end{definition}

\begin{theorem}\label{thm:MAt}
Let $\Omega$ be a bounded open convex set in $\R^d$, and let 
$\vp:\bar\Omega\to\R$ be continuous. 
Assume $\vp$ is locally affine a.e., and let $A$ be the open set defined in \eqref{d:Aset}.
Also assume $\vp$ is non-convex.
Let $t>0$ and define the sets
\[
\calB_t = \nabla\psi_t(A\cap\intr{\Theta_t}), \quad
\calS_t = \nabla\ptss(A\setminus\intr{\Theta_t}). 
\]
Then the Monge-Amp\`ere measure $\mam_t$ for $\pts$ has the 
(Lebesgue) decomposition 
\[
\mam_t = \mu_t+\nu_t \qquad\text{where}\quad 
\mu_t = \lambda\mres\calB_t \,, 
\quad \nu_t = \mam_t\mres \calS_t.
\]
In addition, 
\begin{itemize}
\item[(i)] The sets $\calB_t$ and $\calS_t$ are disjoint,
\item[(ii)] The map $\nabla\psi_t:A\cap\intr{\Theta_t}\to\calB_t$ 
is bijective and locally rigid translation,
\item[(iii)] $\lambda(\calS_t)=0$ and $\mam_t(\calS_t)>0$.
\end{itemize}
\end{theorem}

To proceed toward the proof of the Theorem~\ref{thm:MAt}
we relate $\psi_t^*$ to the function $u_t$ given by the
Hopf-Lax formula (1st Hopf formula)
\begin{equation}\label{e:HLOmega}
u_t(x) = \min_{z\in\bar\Omega} \frac{|x-z|^2}{2t} + \vp(z).
\end{equation}
We relate the touching sets to minimizers in this formula as follows.
First, note that by expanding the quadratic, we have 
\begin{equation}\label{e:pts_ut}
tu_t(x) + \pts(x) = \tfrac12|x|^2  \quad\text{for all $x\in\R^d$.}
\end{equation}
\begin{lemma}\label{lem:HLargmin}
Let $t>0$ and $x\in\R^d$. Then in the Hopf-Lax formula \eqref{e:HLOmega},
a point $y\in\bar\Omega$ is a minimizer if and only if 
$y\in\Theta_t\cap\D\pts(x)$.
\end{lemma}
\begin{proof}
Recall the Young inequality says
\[
x\mdot z \le \psi_t^*(x)+\psi_t^{**}(z)
= \tfrac12|x|^2-tu_t(x) + \ptss(z)
\]
for all $x$ and $z$, with equality 
when $z\in\D\psi_t^*(x)$ or equivalently $x\in\D\psi_t^{**}(z)$.
Since $\ptss\le \psi_t=\frac12|\cdot|^2+t\vp$, we find that for all $z\in\bar\Omega$,
\begin{align*}
tu_t(x) &\le \tfrac12|x-z|^2 -\tfrac12|z|^2 + \ptss(z)
\\& \le  \tfrac12|x-z|^2 + t\vp(z)
\end{align*}
If $z=y$ is a minimizer in \eqref{e:HLOmega} then equality holds in both inequalities here,
hence $\ptss(y)=\psi_t(y)$ and $y\in\D\pts(x)$. And the converse holds:
If $z=y\in \Theta_t\cap\D\pts(x)$ then equality holds in the Young inequality above,
and $\ptss(z)=\psi_t(z)=\frac12|z|^2+t\vp(z)$, and this implies that $y$
is a minimizer in \eqref{e:HLOmega}.
\end{proof}

\begin{lemma}\label{lem:diff}
Let $t>0$. If $y\in \Theta_t\cap\Omega$ and $\vp$ is differentiable at $y$ then 
$\D\ptss(y)$ is a singleton set containing only $x=y+t\nabla\vp(y)$.
\end{lemma}
\begin{proof}
Let $y\in\Theta_t\cap\Omega$. 
Then $\psi_t(z) \ge\ptss(z)$
for all $z$ with equality for $z=y$, 
so given any $x\in\D\ptss(y)$,  it follows
\[
\psi_t(z)-x\mdot z +\tfrac12|x|^2 \ge \psi_t(y)-x\mdot y +\tfrac12|x|^2 
\]
for all $z$ with equality for $z=y$. 
This means that $ \tfrac12|z-x|^2 + t\vp(z) $ is minimized at $z=y$. 
Since $\vp$ is differentiable at $y$, necessarily
$x=y+t\nabla\vp(y)$.
\end{proof}

The touching set $\Theta_t$ is a closed subset of $\bar\Omega$. 
Its (relative) complement is the non-touching set $\Theta_t^c=\bar\Omega\setminus\Theta_t$, 
which is (relatively) open. 
Then their common boundary $\D\Theta_t = \D\Theta_t^c$ is nowhere dense. 

\begin{prop}\label{p:3case}
Let $t>0$, $y\in A$ and $x\in\D\ptss(y)$. Then there are three cases:
\begin{itemize}
\item[(i)] If $y\in \Theta_t^c$ then $\D\pts(x)$ is not a singleton.
\item[(ii)] $y\in\intr{\Theta_t}$\ \ if and only if 
$\D\ptss(y)$ is a singleton set containing $x=y+t\nabla\vp(y)$
and $\D\pts(x)$ is a singleton containing $y$.
\item[(iii)] If $y\in\D\Theta_t$ then 
$\D\ptss(y)$ is a singleton set containing $x=y+t\nabla\vp(y)$
and $\D\pts(x)$ is not a singleton. 
\end{itemize}
\end{prop}
\begin{proof}
1. Suppose $y\in A\cap\Theta_t^c$ and $x\in \D\ptss(y)$. 
Let $y_*\in\bar\Omega$ be a minimizer in the Hopf-Lax formula~\eqref{e:HLOmega}.
Then by Lemma~\ref{lem:HLargmin}, $y_*\in\Theta_t\cap\D\pts(x)$.
But since $y\in\D\pts(x)$ also, $\D\pts(x)$ is not a singleton.
This proves (i).

2. For both parts (ii) and (iii), note that if $y\in A\cap\Theta_t$ then $\vp$ is
differentiable at $y$, so by Lemma~\ref{lem:diff} we have 
$\D\ptss(y)=\{x\}$ with $x=y+t\nabla\vp(y)$.

3. Suppose next that $y\in A\cap\intr{\Theta_t}$.
 Note that in some neighborhood of $y$, $\vp$ is affine
and we have that 
\begin{equation}\label{e:ptssOm}
\psi_t^{**}(z) = \psi_t(z)= \frac12|z|^2+t\vp(z)\,,
\end{equation}
which is strictly convex and quadratic.  
Thus hyperplanes with slope $x$ that support the graph of $\ptss$ at $y$ 
cannot touch it at any other point, so $\D\psi_t^*(x)$ must be a singleton,
and the singleton is $\{y\}$.

4. Now assuming that $y\in A_i$ (so $\nabla\vp(y)=v_i$), that $\D\ptss(y)=\{x\}$ 
where $x=y+tv_i$, and that $\D\pts(x)=\{y\}$, we wish to show $y\in\intr{\Theta_t}$.

By part (i), necessarily $y\in\Theta_t$, and by Lemma~\ref{lem:HLargmin},
$z=y$ is the unique minimizer in the Hopf-Lax formula~\ref{e:HLOmega}.
For any $p\in\R^d$ given, define
\[
H_p(z) := \frac{|p-z|^2}{2t} +\vp(z)\,.
\]
Then $z=y$ is the unique minimizer of $H_x$ in $\bar\Omega$, and $H_x(y)=u_t(x)$. 
Choosing $\delta>0$ so that
$z\in A_i$ whenever $|z-y|<\delta$,  since $H_x$ is continuous
on the compact set $\bar\Omega$ with unique minimizer at $y$,
we necessarily have 
\begin{equation}\label{e:Hxmin}
\min\{H_x(z) :
|z-y|\ge\delta,\,z\in\bar\Omega\} = u_t(x) + \gamma \quad\text{where $\gamma>0$}.
\end{equation}

We claim that if
$|p|>0$ is sufficiently small then $H_{x+p}(z)$ is globally minimized at $y+p$.
By Lemma~\ref{lem:HLargmin} this means $y+p\in \Theta_t$, and $y\in\intr{\Theta_t}$
will follow.
To prove the claim, note that for all $z$, 
\begin{equation}\label{e:Hxp}
H_{x+p}(z) = H_x(z) + \frac{p\mdot(x-z)}t + \frac{|p|^2}{2t} \,.
\end{equation}
Note that $\vp$ takes the form $\vp(z)=v_i\mdot z+h_i$ in the open set $A_i$.
Thus 
we have 
\[
H_{x+p}(z)=\tilde H(z)
\quad\text{for all $z\in A_i$},
\]
where we define $\tilde H$ to be the quadratic function given by
\[
\tilde H(z) := 
 \frac{|x-z+p|^2}{2t} + v_i\mdot z + h_i\quad
\text{for all $z\in\R^d$.}
\]
The global minimum of $\tilde H$ is at $z=x-v_it+p=y+p$.
Provided $|p|<\delta$, this point lies in $A_i$, so the minimum 
of $H_{x+p}(z)$ within $A_i$ 
takes the value 
\[ 
\min_{z\in A_i} H_{x+p}(z)= \tilde H(y+p) = 
H_{x+p}(y+p)=  
H_x(y) + v_i\mdot p \,.
\]
Provided $|v_i||p|<\frac12\gamma$ also, this value
$H_x(y) + v_i\mdot p  < u_t(x)+\tfrac12\gamma$.
On the other hand,
since $x=y+tv_i$,
from \eqref{e:Hxp} and \eqref{e:Hxmin} we find 
that whenever $z\in \bar\Omega\setminus A_i$, 
\begin{align*}
    H_{x+p}(z) \ge u_t(x)+\gamma + \frac{p\mdot(y-z)}t \,. 
\end{align*}
Thus, if $|p|\diam\Omega< \frac12\gamma t$ also, then
\[
H_{x+p}(z) \ge  
u_t(x)+\tfrac12\gamma 
\quad\text{for all $z\in \bar\Omega\setminus A_i$.}
\]
This proves the claim, and finishes the proof of (ii).

5. Part (iii) follows from parts (i) and (ii) as the remaining case.
\end{proof}
Before beginning the proof of Theorem~\ref{thm:MAt} we recall that 
a function $f$, convex on $\R^d$ and finite at $x$, is differentiable at $x$
if and only if $\D f(x)$ is a singleton \cite[Thm. 25.1]{Rockafellar}.

\begin{proof}[Proof of Theorem~\ref{thm:MAt}]
1. On each component $A_i$ of $A$, recall $\nabla\psi_t$ 
is given by rigid translation, $\nabla\psi_t(y)= y+tv_i$.
Moreover, on $A\cap\intr{\Theta_t}$, $\ptss=\psi_t$ is strictly convex
so $\nabla\ptss=\nabla\psi_t$ on this open set and is injective there.
The set $\calB_t=\nabla\psi_t(A\cap\intr{\Theta_t})$ 
is then a disjoint union of open sets 
\[
\calB_t = \bigsqcup_i B_i\,,
\quad\text{where}\quad
B_i=\nabla\psi_t(A_i\cap \intr{\Theta_t}) = (A_i\cap\intr{\Theta_t}) + tv_i.
\]
For each $x\in B_i$, $x=y+tv_i = \nabla\ptss(y)$ where $y\in A_i\cap\intr{\Theta_t}$.
So by part (ii) of Prop.~\ref{p:3case},  
$\D\pts(x)$ is the singleton $\{y\}$, hence $\pts$ is differentiable at $x$ with 
$\nabla\pts(x)= x- t v_i$.
Given any Borel set $B\subset\calB_t$, then 
\[
\mam_t(B)= \sum_i \lambda(\nabla\pts(B\cap B_i)) = 
\sum_i \lambda(B\cap B_i-tv_i) = \lambda(B).
\]
Thus $\mam_t\mres\calB_t=\lambda\mres\calB_t$.
\smallskip

2. For each point $x\in\hat\calS_t:=\D\ptss(A\setminus\intr{\Theta_t})$ we have $x\in\D\ptss(y)$ for 
some $y\in A\setminus\intr{\Theta_t}$.
By parts (i) and (iii) of Prop.~\ref{p:3case}, $\D\pts(x)$ is not a singleton.
Thus $\pts$ is not differentiable at any point of $\hat\calS_t$.
As $\pts$ is convex, hence locally Lipschitz, 
we must have $\lambda(\hat\calS_t)=0$ by Rademacher's theorem.
\smallskip

3. By step 1, $\D\pts$ is single-valued on $\calB_t$.
Since $\D\pts(x)$ cannot be both singleton and non-singleton,
$\calB_t$ and $\hat\calS_t$ are disjoint.
Moreover, 
\[
\D\pts(\calB_t) = A\cap\intr{\Theta_t} \quad\text{ and }\quad
\D\pts(\hat\calS_t)\supset  A\setminus\intr{\Theta_t}.
\]
Since $\D\pts(x)\subset\bar\Omega$ for any $x\in\R^d$ 
and $A$ has full measure in $\bar\Omega$,
\begin{align*}
\mam_t(\R^d) &\le \lambda(\bar\Omega) = 
\lambda(A\cap\intr{\Theta_t})
+\lambda(A\setminus\intr{\Theta_t})
\\ &\le \lambda(\D\pts(\calB_t))+\lambda(\D\pts(\hat\calS_t))
\\ & = \mam_t(\calB_t)+\mam_t(\hat\calS_t) \le \mam_t(\R^d).
\end{align*}
Hence equality holds throughout, whence we get the Lebesgue decomposition
\[
\mam_t = \kappa_t\mres\calB_t+\kappa_t\mres\hat\calS_t.
\]
4.  Let $B=\hat\calS_t\setminus\calS_t$ 
with $\calS_t=\nabla\ptss(A\setminus\intr{\Theta_t})$
and let $\hat A = (\D\ptss)\inv(B)=\D\pts(B)$.
By definition of $\calS_t$, $\ptss$ is not differentiable  at any point of 
$\hat A\cap (A\setminus\intr{\Theta_t})$. 
Further, by step 2, $\pts$ is not differentiable at any point of $B$, so  
$\hat A\cap(A\cap\intr{\Theta_t})$ is empty due to Prop.~\ref{p:3case}(ii).
Hence  $\hat A\cap A = \hat A\cap (A\setminus\intr{\Theta_t})$.
Then each point of $\hat A$ is either a point where $\ptss$ is not differentiable
or is in $A^c$.
Hence
by \eqref{e:ktpullback},
\[
\lambda(\hat A) = 
\mam_t(\hat\calS_t\setminus\calS_t)=0,
\]
whence  $\mam_t\mres\hat\calS_t=\mam_t\mres\calS_t$. 

Moreover, $\mam_t(\calS_t)>0$ since 
the non-touching set $\Theta_t^c$ is relatively open in $\bar\Omega$
and so the open set $A\cap\Theta_t^c\subset A\setminus\intr{\Theta_t}$
is non-empty.
\end{proof}

\begin{remark}
Our results in this section can be compared with work in cosmology by Frisch~\etal\ \cite{frisch2002reconstruction}
and Brenier~\etal\ \cite{brenier2003} that uses the adhesion model for cosmological reconstruction.
    In these works the authors use optimal transportation to determine 
    an initial velocity potential for matter flow in a large region of the universe,
    from presumed mass distributions at two epochs of a time-like variable.
    Without getting into details, the adhesion model takes the velocity potential essentially as
    the viscosity solution $u$ of \eqref{e:HJ1}, the zero-viscosity limit of the potential Burgers equation, and 
    the primordial mass density as uniform. 
    The present distribution of cold dark matter is inferred from observations
    and exhibits concentrations such as mass sheets, filaments and nodes,
    and appears to be taken to correspond to the \MA\ measure $\kappa_t$.
    
    As discussed in \cite{brenier2003}, optimal transport in principle can determine
    only the convexified potential ($\ptss$ here) whose gradient pushes the initial uniform distribution
    forward to $\kappa_t$, and  the original velocity can be inferred only at points
    outside of mass concentrations at the present time.  
    
    In Theorem~\ref{thm:monge_ampere} above, this compares to points in $\calB_t$, 
    the set where the absolutely continuous part of $\kappa_t$ is concentrated.
    Naturally, our assumption that the initial velocity potential is locally affine is
    not suitable for cosmology.
\end{remark} 
\begin{remark} 
    A more general related result exists that describes rigorously how the 
   Lebesgue decomposition  of \MA-like measures is determined in terms of the Alexandrov Hessian of the transport potential. 
   See Remark~7.4 in the lecture notes of Ambrosio~\etal~\cite{AmbrosioBrueSemola2021}.
   One can alternatively prove Theorem~\ref{thm:MAt}
   by using the result of that Remark together with the results of Proposition~\ref{p:3case} above,
   but we retain the arguments above for simplicity.
\end{remark} 

%\pagebreak %s------------------------------------------

\section{Stability and approximation of rigidly breaking flows}
\label{s:stability}

For the rigidly breaking potential flows provided by 
Theorem~\ref{thm_countable}, the countable Alexandrov theorem,
a natural question that arises is whether and in what sense
the flow produced depends continuously on the mass-velocity data,
particularly in the absence of a moment assumption.
In this section we provide a stability theorem that addresses this issue.

Recall from Remark~\ref{r:pp} that sets of mass-velocity data $\{(m_i,v_i)\}$
for which Theorem~\ref{thm_countable} applies are in bijective correspondence
with pure point measures $\nu$ on $\R^d$ having $\nu(\R^d)=\lambda(\Omega)$.
A natural notion of stability of the flows determined by such data 
involves weak-star convergence of measures in $\calM(\R^d)=C_0(\R^d)^*$,
the space of finite signed Radon measures on $\R^d$.

\begin{theorem}\label{thm_stability}
    Let $\Omega\subset\R^d$ be a bounded convex open set with $\lambda(\Omega)=1$.
    For each $n\in\N\cup\{\infty\}$ let $\nu_n$ be a pure point probability measure on $\R^d$. 
   Let $\vp_n$ be the potential 
    associated with $\nu_n$ in the proof of Theorem~\ref{thm_countable}, and
    let $A^n$ be the open set in $\Omega$ given by \eqref{d:Aset} with $\vp_n$ replacing $\vp$.
    Let $X^n_t=\id+t\nabla\vp_n$ be the corresponding flow map, and
    also let $\mam^n_t= \lambda\mres X^n_t(A^n)$. 
    
    If $\nu_n\wkto \nu_\infty$ as $n\to\infty$  weak-$*$ in  $\calM(\R^d)$,
    then 
    $\mam^n_t\wkto \mam^\infty_t$ weak-$*$ in $\calM(\R^d)$ 
    for each $t>0$.
    \end{theorem}

    The basis of the proof is the following result,
    which provides a stability theorem for the transport maps 
    provided by McCann's main theorem in \cite{McCann95}.
    This result is unlikely to be new, but we were unable to locate a precise reference.
    It is closely related to well-known stability results for transport maps in optimal transport theory ---
    see Corollary 5.23 in Villani's book~\cite{Villani_OldNew}, e.g.
    The result of that Corollary does not apply here, however, because
    we make no assumptions regarding optimality or bounded moments for the measures $\nu_n$.

\begin{theorem}\label{thm_McCann}
    Let $\mu$ be a probability measure 
    on $\R^d$ absolutely continuous with respect to Lebesgue measure $\lambda$.
    For each $n\in\N\cup\{\infty\}$, let $\nu_n$ be a probability measure on $\R^d$,
    and let $\vp_n:\R^d\to\R\cup\{\infty\}$ be a convex function as given by McCann's main theorem in \cite{McCann95}.
    If $\nu_n\wkto\nu_\infty$ weak-$*$ as $n\to\infty$,
    then $\nabla\vp_n$ converges to $\nabla\vp_\infty$ in $\mu$-measure on $\R^d$.
\end{theorem}

\begin{proof}
   The coupling defined by
   $\gamma_n=(\id\times\nabla\vp_n)_\sharp\mu$ has marginals $\mu$ and $\nu_n$.
   These couplings are probability measures on $\R^d\times\R^d$, so by the Banach-Alaoglu
   theorem, any subsequence has a further subsequence that converges weak-$*$ to some measure
   $\gamma\in \calM(\R^d\times\R^d)$.  Since we assume that $\nu_n$ converges weak-$*$
   to $\nu_\infty$, by Lemma~9(ii) of \cite{McCann95} we infer that the limit
   measure $\gamma$ is a probability measure coupling $\mu$ and $\nu_\infty$.
   Lemma~9(i) of \cite{McCann95} implies the support of $\gamma$ is cyclically monotone
   in the sense of McCann's Definition 3, hence, as McCann states, a theorem of Rockafellar
   implies that the support of $\gamma$ is contained in the subdifferential of some convex
   function $\psi$ on $\R^d$. Next, by Proposition~10 of \cite{McCann95}, 
   the gradient of $\psi$ pushes $\mu$ forward to $\nu_\infty$, 
   i.e., $\nabla\psi_\sharp\mu=\nu_\infty$. 
   
   By the uniqueness part of McCann's main theorem in \cite{McCann95}, it follows that 
   $\nabla\psi=\nabla\vp_\infty$ $\mu$-a.e.~in $\R^d$.
   Thus we can say the coupling $\gamma = (\id\times\nabla\vp_\infty)_\sharp\mu$.
   Since this limit measure $\gamma$ is unique, the full sequence $\gamma_n$
   converges to it.

   The last step of the proof is to invoke Theorem~6.12 on stability of transport maps 
   in \cite{AmbrosioBrueSemola2021}, which states that
   in this situation, the weak-$*$ convergence of $\gamma_n$ to $\gamma$ is equivalent
   to the convergence of $\nabla\vp_n$ to $\nabla\vp_\infty$ in the sense of $\mu$-measure
   on $\R^d$. This finishes the proof of Theorem~\ref{thm_McCann}.
    \end{proof}

\begin{proof}[Proof of Theorem~\ref{thm_stability}]
Make the assumptions stated in the Theorem. 
  For each $n\in\N\cup\{\infty\}$, the transport map $X^n_t = \id + t\nabla\vp_n$
  is well-defined on the set $A^n$.
   Let $\mu=\lambda\mres\Omega$, and recall from the proof of Theorem~\ref{thm_countable}
   that $\vp_n$ is a potential associated with $\nu_n$ by McCann's main theorem in \cite{McCann95}.
  For any $t>0$ fixed, evidently  it follows from Theorem~\ref{thm_McCann}  that
  $X^n_t$ converges to $X^\infty_t$ in $\mu$-measure as $n\to\infty$.

  Next,
  recall from Proposition~\ref{prop:mu_convex} that the pushforward measure
  \[
  (X^n_t)_\sharp\mu = \lambda\mres X^n_t(A^n)=\kappa^n_t.
  \]
  In order to prove $\mam^n_t$ converges to
  $\mam^\infty_t$ weak-$*$ on $\R^d$, we should prove that for any continuous function
  $f$ on $\R^d$ that vanishes at $\infty$, 
\begin{equation}\label{e:kappalim}
    \int_{\R^d} f(x)\,d\kappa^n_t(x) \to
    \int_{\R^d} f(x)\,d\kappa^\infty_t(x) \quad\text{as $n\to\infty$.}
\end{equation}
Since the measures $\kappa^n_t$ are uniformly bounded in the space $\calM(\R^d)$, 
it suffices to prove this for functions $f$ of compact support.  But in this case we have
\[
    \int_{\R^d} f(x)\,d\kappa^n_t(x)  = \int_\Omega f(X^n_t(z))\,d\lambda(z),
    \quad n\in \N\cup\{\infty\}.
\]
For any subsequence of these quantities, there is a further subsequence along which
$X^n_t$ converges to $X^\infty_t$ a.e.~in $\Omega$. 
We conclude that \eqref{e:kappalim} holds by using
the dominated convergence theorem and the uniqueness of the limit.
\end{proof}

\begin{remark}
   Consider a countably infinite set $\{(m_i,v_i)\}$ of mass-velocity data with 
   $\sum_i m_i = \lambda(\Omega)=1$ and arbitrary $v_i$.  A natural way to approximate 
   the pure point measure $\nu=\sum_i m_i\delta_{v_i}$ is by truncating to a finite sum 
   of Dirac masses and normalizing, taking $\nu_n = \tilde\nu_n/\tilde\nu_n(\R^d)$, where 
   $\tilde \nu_n = \sum_{i=1}^nm_i\delta_{v_i}$ are the partial sums.
   The Alexandrov theorem  (Theorem~\ref{thmAlex}) then can be used to provide 
   the velocity potential $\vp_n$, instead of McCann's theorem which is based on 
   cyclic monotonicity for couplings.  Theorem~\ref{thm_stability} then implies that
   the piecewise-rigidly breaking flows $X^n_t$ converge to $X_t$ in the sense that
   the restricted Lebesgue measures $\lambda\mres X^n_t(A^n)$ converge 
   weak-$*$ to $\lambda\mres X_t(A)$.

   Evidently this still relies on cyclic monotonicity and Rockafellar's theorem,
   however, through the proof of Theorem~\ref{thm_stability} above.  
   It could be interesting to seek a stability proof that avoids this reliance and
   proceeds completely in the spirit of Minkowski and Alexandrov, perhaps 
   using a standard stability theorem for \MA\ measures like 
   Prop.~2.6 in \cite{Figalli-MongeAmpere}. 
\end{remark}

%\pagebreak
%s------------------------------------------
\section{Incompressible optimal transport flows with convex source}\label{s:incompressible}

In this section we complete our characterization of incompressible optimal transport flows
with convex source as was mentioned in the Introduction.
Our paper \cite{LPS19} with Dejan Slep\v cev mainly concerned transport
distance along volume-preserving paths of set deformations.  In terms of optimal transport, 
effectively this means studying paths $t\mapsto \rho_t=\lambda\mres\Omega_t$ comprising 
Lebesgue measure on a family of sets $\Omega_t$ having the same measure.
One of the main results of \cite{LPS19} was that, given two bounded measurable sets
$\Omega_0$ and $\Omega_1$ of equal measure, the infimum of the Benamou-Brenier action
\[
\calA = \int_0^1\int_{\R^d} |v|^2\,d\rho_t\,dt \,,
\]
subject to the transport equation $\D_t\rho+\nabla\cdot(\rho v)=0$, 
but further constrained by the requirement that the measures $\rho_t$ have the form
\begin{equation}\label{c:rho_constraint}
\rho_t=\lambda\mres\Omega_t\,,  \quad t\in [0,1],
\end{equation}
is the {\em same} as $d_W(\mu,\nu)^2$, the
squared  Monge-Kantorovich (Wasserstein) distance  between the measures
\begin{equation}\label{d:munu}
\mu=\lambda\mres\Omega_0, \quad \nu= \lambda\mres\Omega_1.
\end{equation}
The squared distance $d_W(\mu,\nu)^2$ is the infimum of $\calA$ 
{\em without} the constraint~\eqref{c:rho_constraint},  and 
the minimum is achieved for a unique minimizing path $(\mu_t)_{t\in[0,1]}$ 
known as the Wasserstein geodesic path.

Assume $\Omega_0$ and $\Omega_1$ are open, for the rest of this section.
Let $(\mu_t)_{t\in[0,1]}$ be the Wasserstein geodesic path connecting  
the measures $\mu$ and $\nu$ in \eqref{d:munu}.
Theorem~1.4 of \cite{LPS19} says that if the infimum of $\calA$ 
is achieved as described above at some path $(\rho_t)_{t\in[0,1]}$ 
satisfying the constraint \eqref{c:rho_constraint},
then $\rho_t=\mu_t$. That is, any minimizing path satisfying the incompressibility constraint
\eqref{c:rho_constraint} must be the Wasserstein geodesic path.

We refer to such minimizers as {\em incompressible optimal transport paths}. 
Let $(\rho_t)$ be such an incompressible optimal transport path.
Let $\psi$ be the convex Brenier potential whose gradient pushes $\mu=\rho_0$ to $\nu=\rho_1$: 
$\nabla\psi_\sharp\mu=\nu$. Then $\rho_t=(\nabla\psi_t)_\sharp\rho_0$ for each $t\in(0,1)$, 
where 
\begin{equation}\label{e:psitvp}
\psi_t(z) = \frac12|z|^2 + t\vp(z) \quad\text{with}\quad 
\vp(z) =  \psi(z) - \frac12|z|^2.
\end{equation}
At points of differentiabilty of $\psi$, the transport flow is given by 
\[
X_t(z) = \nabla \psi_t(z) = z + t v(z) \quad\text{with}\quad v = \nabla \vp\,.
\]
This velocity potential $\vp$ is semi-convex, by \eqref{e:psitvp}.

Because $\Omega_0$, $\Omega_1$ are bounded open sets and
the characteristic functions on $\Omega_0$ and $\Omega_1$ are smooth,
according to the regularity theory of Caffarelli~\cite{Caffarelli1991}, Figalli~\cite{Fig10} and Figalli and Kim~\cite{FK2010},  $\nabla\psi$ is a smooth diffeomorphism 
$\nabla\psi\colon A_0\to A_1$, where $A_0\subset\Omega_0$ and $A_1\subset\Omega_1$ 
are open sets of full measure.

In this situation, we call the flow given by $X_t$ an {\em incompressible optimal transport flow}
taking $\Omega_0$ to $\Omega_1$.
Corollary~5.8 of \cite{LPS19} states that necessarily the velocity  $v$ of such a flow
is constant on each component of the open set $A_0$ of full measure in $\Omega_0$.
Therefore $\vp$ is locally affine a.e.\ and semi-convex. 

Then the range of $v=\nabla\vp$ is a countable set $\{v_i\}$ of distinct vectors in $\R^d$, 
$v=v_i$ on an open subset $A_i$ with positive measure $m_i=\lambda(A_i)>0$, and  
$\sum_i m_i = \lambda(\Omega_0)$. Recall that we refer to the set $\{(m_i,v_i)\}$ as
the {\em mass-velocity data} of the incompressible optimal transport flow.

\begin{definition}
Let $MV(\Omega_0)$ denote the collection of countable sets of pairs $(m_i,v_i)$ 
such that the $v_i$ are uniformly bounded and distinct in $\R^d$ ($v_i=v_j$ implies $i=j$), 
the $m_i$ are positive, and  $\sum_i m_i = \lambda(\Omega_0)$.  
\end{definition}
As we have just seen, each incompressible optimal transport flow determines
some set of mass-velocity data in $MV(\Omega_0)$. 
The result we are aiming at asserts that 
this association is {\em bijective} if the source domain is convex.

\begin{theorem}\label{thm_optimal}
   Let $\Omega_0$ be a convex bounded open set in $\R^d$.
   Given any incompressible optimal transport flow taking $\Omega_0$ to 
   some other bounded open set, let $\{(m_i,v_i)\}\in MV(\Omega_0)$
   be the mass-velocity data of the flow as described above.
   Then this map from flows to data is bijective.
\end{theorem}

\begin{proof}
Let an incompressible optimal transport flow be given as above, taking
$\Omega_0$ to some bounded open set $\Omega_1$ with the same measure.
Such a flow, and its associated mass-velocity data $\{(m_i,v_i)\}\in MV(\Omega_0)$,
is determined uniquely by the a.e.-locally affine and semi-convex velocity potential $\vp$.
Since $\Omega_0$ is convex, the potential $\vp$ is necessarily convex by Theorem~\ref{thm_semi}. 
Then Theorem~\ref{thm_countable} applies.
Because of the invariance of $\vp$ under reordering of the data as discussed in 
Remark \ref{r:pp},
the set of pairs $\{(m_i,v_i)\}$ 
determines $\vp$ (up to a constant), and hence the flow, uniquely.

Conversely, given any countable set $\{(m_i,v_i)\}$ in $MV(\Omega_0)$,
Theorem~\ref{thm_countable} provides 
velocity potential $\vp$ that is convex and locally affine a.e.\ on $\Omega=\Omega_0$ 
The velocity field $v=\nabla\vp$ defined a.e.\ is bounded, rigidly breaks $\Omega_0$,
and the ensuing flow is an incompressible optimal transport flow.
\end{proof}

%s------------------------------------------
\section{Shapes of shards}\label{s:shape}

In Section~\ref{s:finite}, we have seen that when the number of pieces $A_i$ is finite, 
the pieces are bounded by hyperplanes, like polytopes.  
And in general, with infinitely many pieces possible, the pieces are convex.
It is interesting to investigate what shapes the pieces may have.
In this section we will discuss constructions that show
a given piece may take an arbitrary convex shape, for example, 
or that all pieces can be round balls.

\subsection{Power diagrams} 

Recall that in the case of finitely many pieces, the $A_i$ are
determined by the condition (6).
This means that, with 
$\varphi(x)=v_i\mdot x+h_i$ in $A_i$ as in \eqref{e:affine},
\begin{equation}\label{Ai:1}
A_i = \{x\in \Omega: 
v_i\mdot x + h_i > v_j\mdot x + h_j \mbox{ for all $j\ne i$}\} \,.
\end{equation}
Through completing the square, this provides the equivalent description
\begin{equation}\label{Ai:2}
A_i = \{x\in \Omega: |x-v_i|^2 - w_i < |x-v_j|^2 - w_j \mbox{ for all $i\ne j$}\},
\end{equation}
where $w_i = 2h_i+|v_i|^2$. 
This realizes the decomposition of $\Omega$ into the pieces $A_i$
as a {\em power diagram} determined by the points $v_i$ and weights $w_i$.
Power diagrams are a generalization of Voronoi tesselations 
(for which the $w_i=0$) and which have many uses in computational 
geometry and other subjects, see~\cite{Au87,AuBook}.

In the general case here, when $\varphi$ is convex and locally affine 
a.e.~with countably many pieces possible, the pieces $A_i$ satisfy
\begin{align} \label{Ai:inf1}
A_i &= \interior\bigl\{x\in \Omega: 
v_i\mdot x + h_i \ge \sup_{j\ne i} v_j\mdot x + h_j\,\bigr\} \,,
\end{align}
($\interior$ denotes the interior)
or with $w_i = 2h_i+|v_i|^2$ as before, 
\begin{align} \label{Ai:inf2}
A_i &= \interior\bigl\{x\in \Omega: |x-v_i|^2 - w_i \le \inf_{j\ne i} |x-v_j|^2 - w_j\,\bigr\}\,.
\end{align}
Thus the decomposition of $\Omega$ into the $A_i$ can be considered
as a countable power diagram determined by the countably many points $v_i$
and weights $w_i$.

\subsection{Full packings by balls}\label{ss:full_packing}

The power-diagram description motivates the possibility that with countably
many pieces, the pieces can assume some convex shape different from a
polytope, such as a ball.  We will describe three ways that
optimal breaking can produce pieces that are {\em all} ball-shaped.

Take $\Omega\subset\R^d$ as any bounded open convex set. By a {\em full packing} 
of $\Omega$ by balls we mean a countable collection of disjoint open balls
$B_i=\{x: |x-x_i|<r_i\}$ in $\Omega$ with centers $x_i$ and radii $r_i$, such
that the union $B=\bigsqcup_i B_i$ is an open set of full measure in $\Omega$.

\begin{lemma}\label{l:balls}
Given any full packing $\{B_i\}$ of $\Omega$ by balls, there exists a 
function $\varphi$ convex and locally affine a.e., with 
pieces $A_i=B_i$, such that $\nabla\phi$ maps $B_i$ to the center of $B_i$.
\end{lemma}

\begin{proof}
Since $\bigcup_{j\ne i}B_j$ is dense in $\Omega\setminus B_i$, we can say
\begin{equation}\label{d:Bi}
B_i = \{x\in \Omega: |x-x_i|^2 - r_i^2 <\inf_{j\ne i} |x-x_j|^2 - r_j^2 \}.
\end{equation}
Comparing this with \eqref{Ai:inf2}, we see that the $B_i$ constitute
a power diagram determined by the ball centers $x_i$ and squared radii
$w_i=r_i^2$. We infer that the convex function defined by 
\begin{equation}\label{d:phiV}
\varphi(x) = \sup_i v_i\mdot x + h_i \quad\mbox{ with\ \ } 
v_i = x_i\,, \quad h_i = \frac12(r_i^2-|x_i|^2)\,,
\end{equation}
is locally affine a.e., with 
pieces $A_i=B_i$ and  $\nabla\varphi=x_i$ in $A_i$.
\end{proof}

Any velocity potential $\varphi$ produced by this lemma cannot be $C^1$,
for each point in the set of ball centers $\{x_i\}$ is isolated,
so $\nabla\varphi(\Omega)$ cannot be connected.

Full packings by balls can be produced in a variety of ways. 
Three that are interesting to discuss are
(i) using Vitali's covering theorem;
(ii) so-called {\em osculatory} packing;
(iii) Apollonian packing.

\subsubsection*{(i) Using Vitali's covering theorem}
The collection of all open balls in $\Omega$ constitutes a 
{\em Vitali covering} of $\Omega$, so a full packing of $\Omega$ by balls exists 
by the Vitali covering theorem \cite[Thm. III.12.3]{DS}. 
Actually, one can specify a finite number of the balls at will:
Take $B_1,\ldots,B_k$ to be given disjoint balls in $\Omega$. Then
apply the Vitali covering theorem to the collection of open balls in 
$\Omega\setminus \bigcup_{i=1}^k \bar B_i$.

\subsubsection*{(ii) Osculatory packings}
A sequence $\{B_j\}$ of disjoint balls in $\Omega$ is called {\em osculatory} 
if $B_i$ is a ball of largest possible radius in 
$\Omega\setminus\bigcup_{j=1}^{i-1} B_j$ whenever $i$ is greater than some $k$.
Boyd~\cite{Boyd70} elegantly proved that an osculatory sequence 
in any open set $\Omega\subset\R^d$ of finite measure is a full packing.
Earlier, Melzak \cite{Melzak} had proved this
for the case of dimension $d=2$ and when $\Omega$ itself is a disk.

\subsubsection*{(iii) Apollonian packings of disks}
A classic and beautiful tree construction that produces an osculatory packing in
case $\Omega$ is the unit disk in $\R^2$ is associated with the name of
Apollonius of Perga, who in antiquity classified all configurations of
circles tangent to three given ones.

Start with two circles bounding disjoint disks $B_1$, $B_2$ in $\Omega$, 
tangent to each other and tangent to the unit circle. 
These circles determine two curvilinear triangles. 
At stage 1, inscribe a circle in each of the curvilinear triangles.
These circles bound new disks $B_3$, $B_4$ and divide each curvilinear triangle
into three smaller ones. 
At each subsequent stage we continue by inscribing a circle in each of the 
curvilinear triangles created at the previous stage, adding the disks they bound
to the collection, and subdividing the curvilinear `parent' triangle into three `children.'
From the two triangles and disks we start with at stage $1$, 
upon completing stage $k$ we have $2\cdot 3^k$ disks at stage $k$.

Rearranged in order of decreasing radii, the sequence of disks produced 
in this way is osculatory. A proof that this Apollonian sequence 
produces a full packing of $\Omega$ was provided by Kasner \& Supnick in 1943~\cite{KS}.
The closed set $\bar\Omega\setminus \bigcup_i B_i$, determined by removing
the open disks in an Apollonian packing from the unit disk, is 
known as an {\em Apollonian gasket}. It has measure zero and is nowhere dense.

\begin{figure}%\begin{center}
\centerline{\includegraphics[width=4.2in]
{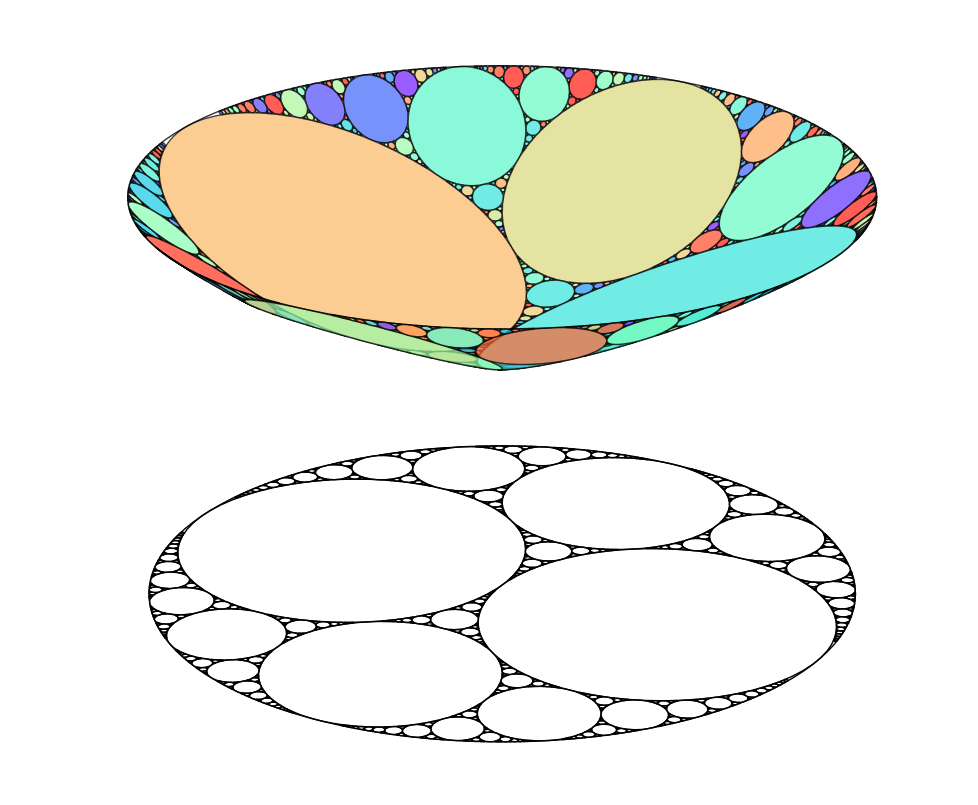}}
%{apollonian_dish2b.png}
\caption{Apollonian bowl: graph of velocity potential locally affine a.e.}\label{f:bowl}
%\end{center}
\end{figure}

Apollonian packings can be generated algorithmically using the generalized Descartes
circle theorem due to Lagarias~\etal~\cite{LagariasEtal02}.  
If parent circles $C_1$, $C_2$, $C_3$ (possibly including the unit circle)
are mutually tangent and tangent to children $C_4$ and $C_5$,
and $C_i$ has complex center $z_i$ and
curvature $b_i=1/r_i$ (with $b_i=-1$ for the outer unit circle),
this theorem implies
\begin{align*}
b_4+b_5 &= 2(b_1+b_2+b_3),\\
b_4z_4+b_5z_5 &= 2(b_1z_1+b_2z_2+b_3z_3).
\end{align*}
From the data $b_i$ and $z_i$ for three parent circles
and one child, these equations determine the entire packing.
Famously, all curvatures $b_i$ are integers if the initial four are.
Possibly, this property was first noticed only in the 20th century
by the chemist Soddy~\cite{soddy1937bowl}. 
In Fig.~\ref{f:bowl} we plot the graph of the convex and a.e.-locally
affine velocity potential generated by Lemma~\ref{l:balls} in this case.
Recall that Fig.~\ref{f:spray} illustrates the rigidly separated disks
$X_t(B_i)$ at time $t=0.5$ as shaded in blue.

\subsection{Shards with arbitrary convex shape}

As promised, we will show here that it is possible 
for some piece to assume an arbitrary convex shape.
Let $\Omega\subset\R^d$ be a bounded open convex set, and let $U$ be 
any convex open subset of $\Omega$.
Without loss of generality, for convenience we translate and scale coordinates 
so that $0\in U$ and $\Omega$ is contained in the unit ball $\{x:|x|<1\}$.

To begin, we construct a sequence of approximations to the distance function
\[
\Phi(x) := \dist(x,U) =\inf\{|x-y|:y\in U\}.
\]
The function $\Phi$ is convex, and of course, 
$\bar U=\{x\in\bar\Omega: \Phi(x)=0\}$.
Let $\{\sigma_i\}_{i\in\N}$ be a sequence of unit vectors in $\R^d$ dense in the sphere 
$\mathbb{S}^{d-1}$, and for each $i$ choose $x_i\in\D U$ to maximize $\sigma_i\mdot x$
on $\bar U$.  Then it is simple to show that 
\begin{equation}
\Phi(x) = 0\vee \sup_{i\in\N} \sigma_i\mdot (x-x_i).
\end{equation}
For each $n\in\N$, put
\[
\Phi_n(x) = 0\vee \max_{1\le i\le n} \sigma_i\mdot (x-x_i). 
\]
Then $\Phi_n(x)$ increases as $n\to\infty$ to the limit $\Phi(x)$ for all $x$, with
\begin{equation}\label{e:Pnest}
0\le \Phi_n(x)\le \Phi(x)\le1. \qquad 
\end{equation}
Moreover, $\Phi_n$ is convex and piecewise affine, and 
since $|\nabla\Phi_n|\le 1$ a.e.~we have
$|\Phi_n(x)-\Phi_n(y)|\le |x-y|$ 
for all $x,y\in\Omega$.
Invoking the Arzela-Ascoli theorem we can conclude that $\Phi_n$ converges uniformly to $\Phi$.
Thus, for any $k\in\N$ there exists $N_k$ such that for all $n\ge N_k$,
\begin{equation}\label{Phiest}
\sup_{x\in\Omega} |\Phi_n(x)-\Phi(x)| <\frac1k \,.
\end{equation}

With these preliminaries, we can construct a convex function, locally affine a.e.,
having $U$ as one of its 
pieces, as follows.
\begin{prop}
Let $\{a_k\}_{k\in\N}$ be a decreasing sequence of positive numbers satisfying
$a_{k+1}\le \frac1k a_k$ for all $k$.  Let
\begin{equation}
\varphi(x) = \sup_k a_k \Phi_{N_k}(x) \,, \qquad x\in\Omega.
\end{equation}
Then $\varphi$ is nonnegative, convex, and locally affine a.e., with $\varphi(x)=0$
if and only if $x\in\bar U$.
\end{prop}

\begin{proof}
Let
$\varphi_k(x) = \max_{1\le i\le k} a_i \Phi_{N_i}(x)$. 
Then $\varphi_k$ is nonnegative and piecewise affine, and it vanishes on a polytope 
containing $U$.  If $\dist(x,U)=\Phi(x)>\frac2k$ and  
$x\in\Omega$, then by \eqref{Phiest} and \eqref{e:Pnest},
\[
\varphi_k(x) \ge a_k \Phi_{N_k}(x) > \frac {a_k}k\ge a_{k+1}\ge a_{k+1}\Phi_{N_{k+1}}(x),
\]
hence $\varphi_{k+1}(x)=\varphi_k(x)$.
By consequence, $\varphi$ is piecewise affine outside any open neighborhood of $\bar U$.
We can conclude it is locally affine a.e.~in $\Omega$ and vanishes only on $\bar U$.
\end{proof}

%s------------------------------------------
\section{Discussion}\label{s:discuss}

In this paper we have focused attention on flows
that rigidly break a convex domain, flows of a type that permits
a classification in terms of mass-velocity data for the pieces.
In particular, we have investigated conditions under which 
rigidly breaking potential flows must arise from a {\em convex} potential. 
As mentioned in Remark~\ref{rem:conjecture},
it may be reasonable to conjecture that the conditions 
(i)-(ii) in Theorem~\ref{thm_main} which ensure the
potential's convexity may be weakened or discarded.
We have also investigated and illustrated 
several differences between flows that break a domain into
finitely many vs. infinitely many pieces.

We conclude this paper with a discussion of a few points, concerning:
(a) conditions that ensure the velocity field can be realized
as the gradient of a continuous potential;
(b) in our one-dimensional example of subsection~\ref{ss:cantor},
the fat Cantor sets expand {\em uniformly} in time;
(c) some necessary criteria for a rigidly breaking velocity field
to be continuous in dimensions $d>1$.

\subsection{Sufficient conditions for continuity of the potential}

In Theorem~\ref{thm_main} we assume 
the velocity field is the gradient of a potential $\vp$
that is locally affine a.e.~in the convex set $\Omega$,
and we assume {\em a priori} that $\vp$ is continuous.
In this subsection we briefly investigate conditions on $v$
that are sufficient to ensure these properties.

In order that  some $\vp\in L^1_{\rm loc}(\Omega)$ should exist with
$v=\nabla\vp$ in the sense of distributions,
it is simple to check that necessarily the distributional Jacobian matrix
$(\D_jv_k)$ should be symmetric. In physical terms, this 
means that the velocity field should generate {\em no shear}.

Some integrability condition on $v$ appears needed as well.
Note, however, that Theorem~\ref{thm_countable}, our countable Alexandrov theorem, 
provides a rigidly breaking velocity field $v$ that fails to be in 
$L^1(\Omega)$ if the mass-velocity data is such that
$\sum_i m_i|v_i|=\infty$. However, since $v=\nabla\vp$ with $\vp$ convex,
necessarily $v$ is {\em locally} bounded a.e.~in $\Omega$.

In order to ensure that a velocity field $v=\nabla\vp$ with $\vp$
continuous, then, we should require $v$ is curl-free
and it is reasonable to require some local boundedness or integrability in $\Omega$.
We find the following conditions are indeed sufficient.

\begin{prop}
    Let $\Omega\subset\R^d$ be bounded, open and convex.
    Suppose that for some $p>d$, $v\in L^p_{\rm loc}(\Omega,\R^d)$ and 
    its (matrix-valued) distributional derivative is symmetric.
    Then $v=\nabla \vp$ a.e.~in $\Omega$,
    for some locally H\"older continuous function $\vp:\Omega\to\R$.
\end{prop}

\begin{proof}[Proof]
   By a standard cutoff and mollification argument we find a sequence of smooth velocity fields
   $v^k$ converging to $v$ in $L^p_{\rm loc}(\Omega)$. Fix $z_0\in\Omega$. Inside any convex subdomain 
   $\Omega'\subset\Omega$ with compact closure in $\Omega$ and containing $z_0$, 
   we can ensure that for $k$ sufficiently large, 
   the $v^k$ are curl-free, having symmetric Jacobian matrices 
   $\nabla v^{k}$ inside $\Omega'$. By path integration along line segments from $z_0$,
   we can define smooth $\vp^k$ on $\Omega$ such that $\vp^k(z_0)=0$
   and on $\Omega'$ we have $\nabla \vp^k=v^k$.  Then the sequence $(\nabla\vp^k)$
   is bounded in $L^p(\Omega')$ and by Morrey's inequality, $(\vp^k)$ is bounded in
   $C^\alpha$ norm on $\Omega'$ for $\alpha=1-d/p$.  Then it follows that $\vp^k$
   converges locally uniformly in $\Omega$ to a H\"older continuous limit $\vp$.
\end{proof}

Finally, we comment on what might happen with rigidly breaking flows if shear is allowed.  
Without the potential flow assumption,
it is easy to imagine a great variety of rigidly breaking flows that appear difficult to classify.
E.g., as a simple example consider $\Omega$ to be the unit ball in $\R^2$,
let $f$ be any function whose graph $x_2=f(x_1)$ disconnects
$\Omega$ in two pieces, and let $v$ be the velocity field 
that sends the upper piece moving rigidly upward and the lower piece 
downward at speed 1. 
If the graph is not a horizontal line, however, then the 
distributional curl of $v$ will be concentrated on the graph and nonzero.  

\subsection{Uniform expansion of the Cantor set}\label{ss:1d_wave}

Here we provide a proof of our comment in subsection~\ref{ss:cantor} 
regarding the uniform expansion of the Cantor set under the 
transported velocity field plotted in Fig.~\ref{f:cantor}. 
This figure plots the Cantor-function velocity $v=c(z)$ {\em vs.}~the transported location 
$x=z+tc(z)=X_t(z)$, which is a continuous and strictly increasing
function of $z$ for $t>0$.  Define this velocity as a function 
of $x\in\R$ and $t\ge0$ by 
\begin{equation}\label{d:f}
f(x,t) = c(z), \quad\mbox{ where } x = z+tc(z).
\end{equation}
(Here $c(z)=0$ for $z\le0$ and $=1$ for $z\ge1$.)
This is the Lax implicit formula for a solution of 
the inviscid Burgers equation $\D_t f+ \D_x(\frac12 f^2) = 0$.
The function $f(\cdot,t)$ is increasing.
As discussed in section~\ref{ss:cantor},
$f(\cdot,t)$ is constant on each component interval of the complement
of the ``expanded'' set
$\calC_t = \{z+tc(z): z\in\calC \}$,
which is a fat Cantor set 
of Lebesgue measure $\lambda(\calC_t)=\lambda(\calC_t)=t$.
Indeed, this set expands Lebesgue measure {\em uniformly}, as we now show.

\begin{prop}\label{prop:Lax} For $t>0$, the function
in \eqref{d:f} is given by 
\[
f(x,t) = \frac1t \int_0^x \one_{\calC_t}(s)\,d\lambda(s). 
\]
Thus $\D f/\D x = 0$ on $\calC_t^c$, and 
$\D f/\D x = 1/t$ 
at each Lesbegue point of $\calC_t$.
\end{prop}

\begin{proof} Fix $t>0$. The function $x\mapsto f(x,t)$ satisfies
a one-sided Lipschitz bound (Oleinik inequality), with a simple proof:
Say $\hat x = X_t(\hat z) > x = X_t(z)$. 
Then \[
\hat x - x = \hat z - z + t(c(\hat z)-c(z)) \ge \hat z - z \,,
\]
hence 
\[
0\le \frac{f(\hat x,t)-f(x,t)}{\hat x -x} = 
\frac{c(\hat z)-c(z)}{\hat x -x} = 
\frac1t \left(1-\frac{\hat z-z}{\hat x - x}\right)
\le \frac 1t.
\]

Since $f$ is increasing in $x$, it is Lipschitz, hence differentiable a.e., whence
$0\le \D f/\D x \le 1/t$. 
We infer from Lebesgue's 
version of the fundamental theorem of calculus that 
\begin{align*}
1 = c(1)= f(1+t,t) = \int_{\calC_t}\frac{\D f}{\D x}(s,t)\,d\lambda(s) \le \frac 1t\lambda(\calC_t) = 1.
\end{align*}
Then indeed $\D f/\D x(\cdot,t)=1/t$ a.e. in $\calC_t$, and
\[
f(x,t) = \frac1t \int_0^x \one_{\calC_t}(s)\,d\lambda(s). 
\]
Moreover this shows $\D f/\D x =1/t$ at every Lebesgue point of $\calC_t$. 
\end{proof}

\begin{remark}
    The function $f$ is in fact the entropy solution to the inviscid Burgers equation with initial data
    $f(x,0)=c(x)$, see \cite[Sec.~3.4]{Evans_book}.
\end{remark}

\subsection{On continuous velocities in multidimensions}

We lack any characterization like the one in Proposition~\ref{p:cacx1d} for
describing rigidly breaking velocity fields that are continuous when $d>1$. 
So here we confine ourselves to discuss some necessary constraints.

Suppose $v=\nabla\varphi$ is rigidly breaking and continuous,
where $\varphi$ is $C^1$, convex and locally affine a.e.~on
a bounded open convex set $\Omega\subset\R^d$. 
Let $\{v_i\}$ be the distinct values of $v$ on the 
open set $A$ in \eqref{d:Aset} where 
$\vp$ is locally affine.  Since $A$ is dense in $\Omega$, necessarily the 
set $\{v_i\}$ is dense in the continuous image $v(\Omega)$, which must be {connected},
as in the case $d=1$ treated in Proposition~\ref{p:cacx1d}.

Recall that for all $t\ge 0$, the flow map $X_t$ is a continuous injection 
from $\Omega$ onto $X_t(\Omega)$.
Indeed, it is a homeomorphism, since the inequality 
proved in Lemma~\ref{lem:injective},
\[
|X_t(z)-X_t(y)|\ge|z-y| \,,
\]
implies the inverse is a contraction.  Brouwer's domain invariance theorem 
(see \cite{Kulpa} or \cite[Sec.~1.6.2]{Tao-Hilbert})
implies $X_t(\Omega)$ is open in $\R^d$.
Topologically $X_t(\Omega)$ is the same as $\Omega$, 
not disconnected in any way nor having ``holes.'' 
Instead it is contractible to a point.
Moreover we can deform $\Omega$ into $v(\Omega)$ through the homotopy defined by
\[
S(x,\tau)= (1-\tau)x + \tau v(x)  \,, 
\]
noting $S$ is continuous on $\Omega\times[0,1]$.
Thus the image $v(\Omega)$ is a limit of homeomorphic images 
$S_\tau(\Omega)=X_t(\Omega)/(1+t)$,
$\tau= t/(1+t)$. 

But we have been unable to determine whether $v(\Omega)$ 
must be homotopy equivalent to $\Omega$, or whether this property, say,
would suffice to ensure $\varphi$ be $C^1$.
The monotonicity of the velocity (as in \eqref{e:vp_monotone})
should be relevant, since for example, the smooth but non-monotone map 
$v(x_1,x_2)=(\cos8x_1,\sin8x_1)$ maps the square $\Omega=(0,1)^2$ 
surjectively onto the unit circle.

\appendix

%s------------------------------------------
%s------------------------------------------
%s------------------------------------------

\section*{Acknowledgements}
We thank Robert McCann for pointing out that Theorem~\ref{thm_semi}
follows from~\cite[Lemma~3.2]{McCann97}. 
We are very grateful to an anonymous referee for
corrections and many detailed suggestions for clarification.
Thanks go also to the Isaac Newton Institute for
Mathematical Sciences, Cambridge, for support and hospitality during the
programme Frontiers in Kinetic Theory, where work on this paper was undertaken. 
This work was supported by EPSRC grant no EP/R014604/1.
This material is based upon work supported by the National Science Foundation
under grants DMS 2106988 (JGL) and 2106534 (RLP).

%\nocite{*}
%\bibliographystyle{}
\bibliographystyle{siam}
% \bibliography{broken_only.bib}
\bibliography{broken.bib}

\end{document}